\documentclass[11pt]{amsart}

\usepackage{amssymb, amsmath}

\usepackage{tikz, tikz-cd}
\usepackage{enumitem}
\usepackage{mathtools}
\usepackage{todonotes}

\usepackage{hyperref}
\hypersetup{
    colorlinks=true,
    citecolor=blue,
    linkcolor=blue,
    filecolor=magenta,      
      urlcolor=blue,
}

\theoremstyle{plain}
 \newtheorem{thm}{Theorem}[section]
 
 \newtheorem{cor}[thm]{Corollary}
\newtheorem{lem}[thm]{Lemma}
\newtheorem{prop}[thm]{Proposition}

\theoremstyle{definition}

\newtheorem{rmk}[thm]{Remark}

\theoremstyle{plain}
\newtheorem{theorem}[thm]{Theorem}

\newtheorem{lemma}[thm]{Lemma}
\newtheorem{corollary}[thm]{Corollary}
\newtheorem{proposition}[thm]{Proposition}
\newtheorem{remark}[thm]{Remark}
\newtheorem{observation}[thm]{Observation}

\theoremstyle{definition}

\newtheorem{criterion}[thm]{Criterion}

\newtheorem{defin}[thm]{Definition}

\newtheorem{facts}[thm]{Facts}

\numberwithin{equation}{section}
\newcommand{\sB}{{\mathcal B}}
\newcommand{\sC}{{\mathcal C}}

\newcommand{\sG}{{\mathcal G}}
\newcommand{\sH}{{\mathcal H}}

\newcommand{\sL}{{\mathcal L}}

\newcommand{\sN}{{\mathcal N}}
\newcommand{\sO}{{\mathcal O}}

\newcommand{\NN}{\ensuremath{\mathbb{N}}}
\newcommand{\hol}{\ensuremath{\mathcal{O}}}

\newcommand\om{\omega}
\newcommand\la{\lambda}
\newcommand\Lam{\Lambda}
\newcommand\s{\sigma}

\newcommand\e{\epsilon}
\newcommand\al{\alpha}
\newcommand\be{\beta}
\newcommand\Ga{\Gamma}
\newcommand\De{\Delta}
\newcommand\ga{\gamma}
\newcommand\de{\delta}

\newenvironment{dedication}
        {\begin{quotation}\begin{center}\begin{em}}
        {\par\end{em}\end{center}\end{quotation}}
\DeclareMathOperator{\Pic}{Pic}

\DeclareMathOperator{\Alb}{Alb}

\DeclareMathOperator{\Tors}{Tors}

\newcommand{\HH}{\ensuremath{\mathbb{H}}}

\newcommand{\ra}{\ensuremath{\rightarrow}}

\newcommand{\CC}{\mathbb{C}}
\newcommand{\PP}{\mathbb{P}}
\newcommand{\QQ}{\mathbb{Q}}

\newcommand{\ZZ}{\mathbb{Z}}

\newcommand{\Aut}{\mathrm{Aut}}

\newcommand{\Diff}{\mathrm{Diff}}

\newcommand{\id}{\mathrm{id}}

\newcommand{\Inn}{\mathrm{Inn}}

\newcommand{\Ker}{\mathrm{Ker}}

\newcommand{\Out}{\mathrm{Out}}

\newcommand{\red}{\mathrm{red}}

\newcommand{\tor}{\mathrm{tor}}

\usepackage{listings} 
 \lstdefinelanguage{Magma}%
  {%
   otherkeywords={:=,+:=,-:=,*:=},%
   procnamekeys={function,func,intrinsic,procedure,proc,return},%
   morekeywords={true,false},%
   morekeywords=[2]{adj,and,cat,cmpeq,cmpne,diff,div,eq,ge,gt,in,is,join,le,lt,%
          meet,mod,ne,notadj,notin,notsubset,or,sdiff,subset,xor},%
   morekeywords=[3]{assigned,break,by,case,catch,continue,declare,default,%
          delete,do,elif,else,end,eval,exists,exit,for,forall,fprintf,if,local,%
          not,print,printf,quit,random,read,readi,repeat,restore,save,select,%
          then,time,to,try,until,vprint,vprintf,vtime,when,where,while},%
   morekeywords=[4]{clear,forward,freeze,iload,import,load},%
   morekeywords=[5]{assert,assert2,assert3,error,require,requirege,requirerange},%
   morekeywords=[6]{car,comp,cop,elt,ext,frac,hom,ideal,iso,lideal,loc,map,%
          ncl,pmap,quo,rec,recformat,rep,rideal,sub},%
      sensitive,%
      morecomment=[l]//,%
      morecomment=[s]{/*}{*/},%
      morestring=[b]"%
  }[keywords,procnames,comments,strings]%

\lstnewenvironment{code_magma}[1][]
{\lstset{basicstyle=\scriptsize\ttfamily, columns=fullflexible,
language=Magma,
keywordstyle=\color{red}\bfseries,
commentstyle=\color{blue},tabsize=6,
numbers=left, numberstyle=\tiny,
stepnumber=1, numbersep=5pt, frame=single, #1}}{}

\begin{document}
\title[
 Cohomologically
trivial automorphisms of elliptic surfaces-I]{ 
 On the
cohomologically trivial automorphisms of  elliptic surfaces  I: $\chi(S)=0$. }
\author[F. Catanese]{Fabrizio Catanese}
\author[D. Frapporti]{Davide Frapporti}
\author[C. Glei\ss ner]{Christian Glei\ss ner}
\author[W. Liu]{Wenfei Liu}
\author[M. Sch\"utt]{Matthias Sch\"utt}
\makeatletter
\let\@wraptoccontribs\wraptoccontribs
\makeatother

\date{\today}

\address{Lehrstuhl Mathematik VIII, Mathematisches Institut der Universit\"{a}t
Bayreuth, NW II, Universit\"{a}tsstr. 30,
95447 Bayreuth, Germany, and Korea Institute for Advanced Study, Hoegiro 87, Seoul, 
133-722, Korea}
\email{Fabrizio.Catanese@uni-bayreuth.de}
\address{Politecnico di Milano, Dipartimento di Matematica, via Bonardi 9,
20133 Milano, Italy}
\email{davide.frapporti@polimi.it}
\address{Lehrstuhl Mathematik VIII, Mathematisches Institut der Universit\"{a}t
Bayreuth, NW II, Universit\"{a}tsstr. 30,
95447 Bayreuth, Germany}
\email{Christian.Gleissner@uni-bayreuth.de}
\address{School of Mathematical Sciences, Xiamen University, Siming South Road 422, Xiamen, Fujian Province, P.R. China}
\email{wliu@xmu.edu.cn}
\address{Institut f\"ur Algebraische Geometrie, Leibniz Universit\"at
  Hannover, Welfengarten 1, 30167 Hannover, Germany\\ and\;\;\;\;\;\;\;\;\;\;\;\;\;\;\;\;\;\;\;\;\;\;\;\;\;\;\;\;\;\;\;\;\;\;\;\;\;
  \linebreak
  Riemann Center for Geometry and Physics, Leibniz Universit\"at
  Hannover, Appelstrasse 2, 30167 Hannover, Germany}
\email{schuett@math.uni-hannover.de}

%\thanks{ }
\begin{abstract}
 In this first part we  describe  the group $\Aut_{\ZZ}(S)$ of  cohomologically trivial automorphisms of a properly elliptic surface (a minimal surface $S$ with Kodaira dimension $\kappa(S)=1$),   in the initial   case  $ \chi(\sO_S) =0$.

In particular,   in the   case where $\Aut_{\ZZ}(S)$ is finite,  we give   the  upper bound 4 for its cardinality, showing 
more precisely that  if $\Aut_{\ZZ}(S)$
is nontrivial, it is 
 one of the following groups: $\ZZ/2, \ZZ/3, (\ZZ/2)^2$. We also show with easy examples that the groups
$\ZZ/2, \ZZ/3$ do effectively occur.  

  Respectively,  in  the case where $\Aut_{\ZZ}(S)$ is infinite, 
we give the sharp upper bound 2  for the number  of its  connected components.

\end{abstract}
\maketitle

\begin{dedication}
Dedicated to Yurii  (Gennadievich) Prokhorov on the occasion of his 60th birthday.
\end{dedication}

\tableofcontents
\vspace{-1cm}
\section{Introduction}

Let $X$ be a 
 compact connected
 complex manifold. Bochner and Montgomery \cite{bm1, bm2} showed that the automorphism group $\Aut(X)$
(the group of biholomorphic maps
$ g\colon X \ra X$, i.e., the group of  diffeomorphisms  
$ g \in \Diff (X)$
which preserve the complex structure of $X$) 
is a finite dimensional complex Lie Group, possibly with infinitely many connected components, whose Lie Algebra is the space
$H^0(X,  \Theta_X)$ of holomorphic vector fields on $X$.

 $\Aut^0(X)$ denotes  the connected  component of the identity in $\Aut(X)$,
 and an important   object of interest is the group of components $ \Aut(X) / \Aut^0(X)$,
 which is in general at most countable.

For applications to Hodge Theory and to period mappings  quite important is the study of some larger subgroups,
the group of numerically trivial automorphisms\footnote{The reader should be aware that  some  
authors called the numerically trivial automorphisms   {\bf automorphisms acting trivially on cohomology} 
without specifying that they meant cohomology with rational (and not integral) coefficients.}  
\[
\Aut_\QQ(X):=\{\sigma\in\Aut(X) \mid \sigma \text{ induces the trivial action on }H^*(X, \QQ)\},
\]
and the group of cohomologically trivial automorphisms 
\[
\Aut_\ZZ(X):=\{\sigma\in\Aut(X) \mid \sigma \text{ induces the trivial action on }H^*(X, \ZZ)\},
\]
in  
which we are especially interested in this article.

 It is an easy but important remark to observe that $\Aut_\ZZ(X) = \Aut_\QQ(X)$ if 
$H^*(X, \ZZ)$ has no torsion. In the case where $dim(X)=2$, this amounts by Poincar\'e duality 
(see Remark \ref{poincare}) to the condition   that $H_1(X, \ZZ)$ 
 be a torsion free abelian group.

The case where $X$ is a cKM = compact K\"ahler Manifold, endowed with a K\"ahler metric $\om$,  was
considered around 1978 by Lieberman \cite{Li78} and Fujiki \cite{fujiki}, in particular Lieberman \cite{Li78} proved
that $\Aut_\QQ(X) / \Aut^0(X)$ is a finite group. 

For surfaces $S$ not of general type\footnote{For those of general type $\Aut(S)$ is a finite group and  
$|\Aut_\QQ(S) |$ is universally bounded. 
In fact, $|\Aut_\QQ(S)|\leq 4$ if $\chi(\sO_S)\geq 189$ (\cite{Cai04}). On the other hand, minimal surfaces of general type with $\chi(\sO_S)<189$ form a bounded family, and hence their full automorphism groups have a uniform bound.
 It is an open question whether there are surfaces of general type with $|\Aut_\QQ(S)| > 25.$
}
the aim is to describe the group of numerically trivial automorphisms,
this was begun in 1972 by   Pyatetskij-Shapiro and  Shafarevich \cite{ps-shaf} (see also Burns and Rapoport \cite{br}) who proved that, for a K3 surface $S$, $\Aut_\QQ(S) $ is a trivial group.
 Peters \cite{Pe79}, \cite{Pe80} began the study of $\Aut_\QQ(S) $ for compact K\"ahler surfaces. 
Automorphisms of surfaces were also investigated by Ueno \cite{ueno} and  Maruyama \cite{Ma71}
in the 70's, then Barth and Peters \cite{barth-peters},  Mukai and Namikawa \cite{MN84}, 
 Jin-Xing Cai
\cite{Cai09} and Mukai \cite{mukai}.

We are considering the   case
where $S$ is minimal and the Kodaira dimension 
 $\kappa(S) =1$
 (these are the so-called properly elliptic surfaces).
 
For these, we have a pluricanonical elliptic fibration $ \Phi : S \ra B$, and since $K^2_S=0$,
by the Noether formula the basic invariant is  $\chi(S) : = p_g(S) - q(S) + 1$
such that $\chi(S) \geq 0$ (observe that the topological Euler number $e(S)$ is then determined by
Noether's formula, asserting that $ 12 \chi(S) = e(S)$).

 A first easy but  important tool is to observe that the group $Aut(S)$ preserves this elliptic fibration.
 
 Our aim is to give upper bounds for the cardinality of the group $\Aut_\ZZ(S)$  (respectively  $\Aut_\ZZ(S) / Aut^0(S)$
 if $S$ admits vector fields) in terms of $\chi(S)$,
 hence in this first part we begin with the initial case $\chi(S)=0$.

 If $\chi(S)=0$ the surfaces, if minimal,  are isogenous to a higher  elliptic  product, in the sense of
 \cite{isogenous} and  \cite[Chapter I]{time}, where  the following  fact is proven:
  \footnote{ One says instead, see  \cite{isogenous}, that $S$ is isogenous to a higher product if 
  $S$ is isogenous to the product $C_1 \times C_2$
  of two curves of respective genera at least $2$.}

\begin{lem}\label{lem: k1 str}
Let  $S$ be a minimal surface with  Kodaira dimension $\kappa(S)=1$ and with $\chi(\sO_S)=0$.

 Then 
 $S$ is isogenous to a higher  elliptic product, meaning that:
\begin{enumerate}[leftmargin=*]
    \item $S=(C\times E)/\Delta_G$, where 
  $C$
 and $E$ are smooth curves with ${\rm genus}(C)\geq 2$ and ${\rm genus}(E)=1$, and $G$ is a finite group acting faithfully on $C$ and $E$ such that the diagonal $\Delta_G\subset G\times G$ acts freely on $C\times E$. 
    \item   Moreover     
    \begin{itemize}
    \item[(I)] either  $G$ acts on $E$ by translations,  equivalently
    $$p_g(S)={\rm genus}(C/G) =q(S)-1 \Leftrightarrow  {\rm genus}(E/G)=1 \Leftrightarrow Aut^0(S) \neq \{Id_S\}$$ 
      \item[(II)]
    or  $G$ does not act on $E$ only by translations,   equivalently
    $$p_g(S)=q(S)-1 ={\rm genus}(C/G)-1 \Leftrightarrow {\rm genus}(E/G)=0 \Leftrightarrow Aut^0(S)=  \{Id_S\}.$$
    \end{itemize}
    \item
    The fibres of $f : S= (C\times E)/\Delta_G \ra B = C/G$ are smooth elliptic or multiples of
    a smooth elliptic curves, that is $f$ is an elliptic quasi-bundle.

\end{enumerate}

\end{lem}

  In case (I) of the previous Lemma one says that $S$ is  a {\bf pseudo-elliptic surface}.

If $S=(C\times E)/\Delta_G$, and the action is free, one can assume without loss of generality that $G$ acts faithfully
on $C$ and on $E$, going to a minimal realization (see Remark \ref{faithful}).

 The key problem now is whether $ \Aut_{\ZZ}(S) = \Aut^0(S)$.

 Our main results, summarizing and reformulating as follows Theorems \ref{pseudo-elliptic} and  \ref{lem: act trivial 1}, 
 correct cases (ii) and (iii)  from 
 \cite{Pe80}
which correspond to cases (II) and (I) of the previous Lemma \footnote{  In case (I) the flaw is an unproven assertion  in the proof of Proposition 5.3,
when treating case (d) on page 265, lines 7- 10. In case (II) the analysis is more delicate, as shown in this paper. But, roughly speaking, the major flaw is to assert that if a branch point (downstairs)  is fixed, then also a ramification point  (upstairs) is fixed:
because  $E \ra E/G$ is the composition of an unramified with  a fully ramified morphism. } .

\begin{thm}\label{thm: Main-Theorem-2}
\label{thm}
 Assume that $S$ is a properly elliptic surface ($S$ is minimal of Kodaira dimension $1$),
with $\chi(S)=0$. 

Then $S= (C \times E)/G$ is isogenous to a higher elliptic product and

\begin{enumerate}
\item[(I)]
In the  pseudo-elliptic case where $\Aut^0(S)$ is infinite, equivalently $G$ acts by translations on $E$,
  the subgroup $ \Aut_{\ZZ}(S)$ coincides with $\Aut^0(S)$, except for the pseudo-elliptic  surfaces 
 such that
 \begin{enumerate}
\item[(I-a)]   $G= \ZZ/2m$, where  $m$ is an odd integer, 
 
\item[(I-b)] $C/G= \PP^1$ and $ C \ra \PP^1$ is branched in four points
 with local monodromies $\{ m,m, 2, -2\}$: 
 
 for these we  have $ |\Aut_{\ZZ}(S)/ \Aut^0(S)| = 2$. 
 \end{enumerate}
\item[(II)]
If $\Aut^0(S) = \{ Id_S\}$ but  $ \Aut_{\ZZ}(S) \neq \{ Id_S\}$, then
\begin{enumerate}
\item[(II-a)] $B : = C/G$ has genus $h \geq 1$.
\item[(II-b)]  $\Aut_\ZZ(S)$ is    isomorphic to  one of the following groups:
\[
 \ZZ/2, \ \ZZ/3, \ (\ZZ/2)^2.
\]
\item[(II-c)] 
The first two cases:   $\Aut_\ZZ(S) = \ZZ/2$, and $\Aut_\ZZ(S) = \ZZ/3$,
do actually occur. 
 \end{enumerate}
\end{enumerate}
\end{thm}

 Theorem \ref{pseudo-elliptic} is devoted to case (I),  the pseudo-elliptic case where $G$ acts on $E$ by translations: there we describe 
 $\Aut_\QQ(S) $ explicitly,  and show that the  single notable exception
 in which  $\Aut^0(S)$   has index two in $\Aut_\ZZ(S)$ does indeed exist: we just take the Galois covering 
 described in (I-b) of the above theorem (which exists by Riemann's existence theorem) and show that the involution of $\PP^1$ exchanging the first two branch points
 of $ C \ra B=\PP^1$, and leaving the other two fixed, lifts to an automorphism of $S$.
 
 The hard part of the proof is to show that this automorphism is cohomologically trivial: for showing this we have to
 perform calculations on the universal cover of $S$, which is the universal cover $\HH \times \CC$ of $C \times E$.

 Observe that Theorem 3.1  of \cite{CatLiu21}  uses surfaces in this class to produce examples, for all $n \in \NN$,
 of pseudo-elliptic surfaces such that the index of $\Aut_\ZZ(S)$ in $\Aut_\QQ(S)$ is at least $n$. 
 
 Moreover, Theorem 4.1 of the same paper shows that if we drop the assumption of minimality, then there are blow ups of pseudo-elliptic surfaces such that $\Aut^0(S)$ becomes trivial, but $\Aut_\ZZ(S)$ does not: 
 and we get examples, for all  $n \in \NN$, where $|\Aut_\ZZ(S)| = n$. 
 
 \bigskip
 
Case (II) where $S$ is minimal with $\chi(S)=0$, and  where $G$ does not act on $E$ via translations is dealt with in Theorem 
 \ref{lem: act trivial 1}.

 We have several concrete examples of automorphisms in $\Aut_\ZZ(S)$ as stated  in (II-c): the simplest 
 is one of order $3$ and where the cohomology of $S$ is torsion free.

 We illustrate this  main example in order to give a flavour of the arguments  used.
 
 Firstly, a finite group $G$ acting faithfully on an elliptic curve $E$ is a semi-direct product
 $ G = T \rtimes \mu_r$, where $\mu_r$ is the group of $r$-th roots of unity,  $ r \in \{2,3,4,6\}$,
 and  $T$ is a finite group of translations on $E$, hence $2$-generated.
 
 To give instead the action of $G$ on a curve $C$, it is equivalent to construct a Galois $G$-covering ($C \ra B$)
 of a curve $B$, determined (in view of the Riemann existence theorem) by a homomorphism onto $G$ of 
 the orbifold fundamental group,
 a quotient of  the fundamental group $\pi_1(B \setminus \sB)$ of
 the complement in the curve $B$ of the set $\sB$ of branch points.
 
 In this simple example there are no branch points, $B$ has genus $2$, and the group $G \cong \ZZ/3 \times \mu_3$.
 What is left is to give the monodromy, a surjection $\pi_1(B) \ra G$, and we choose that the first two generators
 $\al_1, \be_1$ map to the two generators of $G$, while $\al_2, \be_2$ map to the identity.
 
 In this way $S \ra B$ is a fibre bundle with fibre $E$ and it turns out that $S$ is without torsion,
 indeed $H_1(S, \ZZ) \cong H_1(B, \ZZ) $. It suffices then to take the group $H \cong \ZZ/3$ of automorphisms of $C \times E$
 acting trivially on $C$ and via translations by elements of $T$ on $E$: these automorphisms induce automorphisms of $S$
 and clearly act trivially on 
  $H^*(C \times E, \QQ)$: hence a fortiori they act trivially on the subspace of $G$-invariants,
  $ H^*(C \times E, \QQ) ^G = H^*(S, \QQ) = H^*(S, \ZZ)$.
 
 We have more complicated examples (where $H_1(S, \ZZ)$ has a nonzero torsion subgroup), and in order to prove Theorem 
  \ref{lem: act trivial 1} we need to run a long list of possible
 groups $G$ and monodromy homomorphisms.

Finally, for applications to the  Teichm\"uller space (stack) of $X$ (\cite{handbook, Me19}) 
the primary object of interest is the subgroup of 
 $C^\infty$-isotopically
 trivial automorphisms
$$\Aut_*(X) : = \{\sigma\in\Aut(X) \mid \sigma \in \Diff^0(X)\},$$
($\Diff^0(X)$
 denotes the connected  component of the identity in the  group of diffeomorphisms),
clearly contained in the group of homotopically trivial automorphisms
\[
\Aut_\sharp(X) = \{\sigma\in\Aut(X) \mid \sigma \text{ is homotopic to }\id_X\}.
\]
 It is an intriguing question to find examples where these two subgroups differ.

The chain of subgroups 
$$
   \Aut^0(X) \vartriangleleft \Aut_*(X) \vartriangleleft \Aut_\sharp(X)\vartriangleleft \Aut_\ZZ(X)
   \vartriangleleft \Aut_\QQ(X)  \vartriangleleft \Aut(X),
$$
 was discussed in our previous paper \cite{CatLiu21}, especially dedicated to the case of complex dimension $n=2$.
  In the present  paper we prove in Proposition \ref{homotopic} that, for surfaces with Kodaira dimension $1$ and $\chi(S)=0$,
 $\Aut_\sharp(S) =   \Aut^0(S)$.

\bigskip

The sequel to this paper, Part II \cite{cls},  deals with cohomologically (and numerically) trivial automorphisms of properly elliptic surfaces $S$
with $\chi(S)>0$,
showing in particular that $|\Aut_\QQ(S)|$ is unbounded:   this  is a surprising result since for 
over  40 years it was believed  that  these groups are trivial for $p_g >0$ \footnote{ Theorem 4.5 of \cite{Pe80}, confirmed by  Cor. 0.3 of \cite{Cai09}) }.

{\bf Acknowledgement:} We want to thank heartily the referee, for providing many suggestions in order to make the 
presentation  of our arguments more transparent and accessible to the reader (repetita juvant, said the Romans).

\section{Preliminaries}
\subsection{Cohomology classes of multiple fibres}
\begin{lem}\label{lem: mult fib qf=0}
Let $f\colon S \rightarrow B$ be a fibred surface with $q_f:=q(S)-g(B)=0$. Then for two distinct multiple fibres $m_1 F_1$ and $m_2 F_2$, we have $[F_1] \neq [F_2]\in H^2(S, \ZZ)$.
\end{lem}
\begin{proof}
Suppose on the contrary that $[F_1] = [F_2]\in H^2(S, \ZZ)$. Then by the long exact sequence associated to the exponential sequence $0\rightarrow \ZZ \rightarrow \sO_S\rightarrow \sO_S^*\rightarrow 0$, we have $\sO_S(F_1-F_2)\in \Pic^0(S)$. Now $q_f=0$ implies that the map $f_*\colon \Alb(S)\rightarrow \Alb(B)$ is an isomorphism. Therefore, its dual $f^*\colon \Pic^0(B)\rightarrow \Pic^0(S)$ is also an isomorphism. Thus there is some $L\in \Pic^0(B)$ such that $f^*L = \sO_S(F_1-F_2)$. On the other hand, since $\sO_S(F_1-F_2)|_{F_1}\not\cong \sO_{F_1}$, it cannot be of the form $f^*L$, which is a contradiction.
\end{proof}

\begin{cor}\label{cor: mult fib qf=0}
Let $f\colon S\rightarrow B$ be a fibred surface with $q(S)=0$. Then for two distinct multiple fibres $m_1 F_1$ and $m_2 F_2$, we have $[F_1] \neq [F_2]\in H^2(S, \ZZ)$.
\end{cor}

\begin{rmk}
If $q_f>0$, it  is  possible that $[F_1] = [F_2]\in H^2(S, \ZZ)$, see Theorem \ref{pseudo-elliptic}. However, $F_1$ and $F_2$ cannot be linearly equivalent.
\end{rmk}

The following lemma is Principle 3, page 185 of \cite{CatLiu21}:

\begin{lemma}\label{m=2}
Let $f\colon S \rightarrow B$ be a fibred surface, let $F'' = m F'$ be a multiple fibre with $F'$ is irreducible,
and let $\s$ be a cohomologically trivial automorphism of $S$: then $\s(F')= F'$ unless possibly if the genus of $B$ is at most $1$, $m=2$, there are only two multiple fibres with multiplicity $2$, they are isomorphic to each other, and all the other multiple fibres have odd multiplicity.
\end{lemma} 

\subsection{Integral 
first homology
of an elliptic surface which is not a quasi-bundle}

Here we say,  as customary,   that an elliptic surface is  a quasi-bundle (as in \cite{Ser91})   if,  for all fibres $F$, $F_{\red}$ is  smooth.

In the contrary case, we have the following Lemma which is essentially known (see \cite{Ser91}, \cite{cz}, but we state it and sketch its proof for the sake of  clarity).
 
\begin{lem}\label{no-torsion} 
Let $f : S \ra B$ be an elliptic surface with a singular fibre which is not a smooth multiple fibre
(this is equivalent to the assumption that $e(S) >0$, and in particular this holds if $f$ admits a section and has a singular fibre).

Then   $H_1(S)$ is the abelianization of the orbifold fundamental group of the fibration,
$$H_1(S) = H_1^{orb} (f) : =  H_1(B, \ZZ) \bigoplus \left( \bigoplus_i (\ZZ/m_i)  \ga_i )\right)/ \ZZ \left(\sum_i\ga_i\right).$$
\end{lem}
\begin{proof}
Let $g$ be the genus of $B$. Then
we have the orbifold fundamental group exact sequence for $f$, see for instance \cite{barlotti}:
\begin{equation}\label{fund-gr-ex-seq} \pi_1(F) \ra \pi_1(S) \ra \pi_1^{orb}(f) \ra 1,
\end{equation}
where, $m_1, \dots, m_r$ being the multiplicities of the multiple fibres, and $ \ga_1, \dots, \ga_r$ 
being simple geometric loops around the branch points,
\begin{multline*}
 \pi_1^{orb}(f) : = 
 \langle \al_1, \beta_1, \dots, \al_g, \beta_g , \ga_1, \dots, \ga_r \mid \\
\prod_i [\al_i, \beta_i]  \ga_1 \dots \ga_r =1,   \ga_1^{m_1}=  \dots =  \ga_r^{m_r} = 1\rangle .
\end{multline*}

\begin{remark}\label{abelianization}
In general, to an exact sequence of groups 
$$  A \ra B \ra C \ra 1$$
corresponds another exact sequence
$$  A^{ab}  := A / [A,A]   \ra B^{ab}  := B / [B,B]  \ra C^{ab} : = C/[C,C] \ra 0.$$
\end{remark}
\bigskip

From this remark and equation \eqref{fund-gr-ex-seq}
we deduce an exact sequence    
 $$ H_1 (F, \ZZ) = \pi_1(F) \ra H_1(S, \ZZ) \ra (\pi_1^{orb}(f) )^{ab} = H_1^{orb} (f)  \ra 0 $$

The assumption made on the existence
of a singular fibre which is not multiple of a smooth curve, and Kodaira's analysis of local monodromies
implies by \cite[Lemma 1.39]{cz} that the image of $ H_1 (F, \ZZ) $  is zero in $ H_1(S, \ZZ)  $, hence 
it  follows then that $ H_1(S, \ZZ) \cong (\pi_1^{orb}(f) )^{ab}=  H_1^{orb} (f) $
 as claimed.
\end{proof}

\begin{remark}\label{poincare}
The previous lemma determines in particular  the torsion subgroup of $H_1(S, \ZZ)$, which, by Poincar\'e's 
duality, is also the torsion subgroup of  $H^3(S, \ZZ)$.

Moreover  there is a canonical isomorphism of the respective
torsion subgroups of $H^2 (S, \ZZ) \cong H_2(S, \ZZ)$ and of $H_1(S, \ZZ)$ (see \cite[ Thm.\ 23.18]{greenberg}). 

 Since this isomorphism is  canonical, and is associated to the complexes of groups of singular chains and cochains, 
on which any  homeomorphism acts compatibly, an automorphism which  acts as the identity on the torsion subgroup of $H_1(S, \ZZ)$
will also act as the identity on the torsion subgroup of $H_2(S, \ZZ)$.

 Finally, to verify that an automorphism is cohomologically trivial, it suffices to verify
that it acts as the identity on $H_1(S, \ZZ)$ and on  $H^2 (S, \ZZ)$.

\end{remark}

The next section takes, in particular, care of the case where,  for all fibres $F$, $F_{\red}$ is  smooth.

\subsection{Further notation} Given a variety $X$, we shall often denote $\pi_1(X)$ simply by $\pi_X$.

The fundamental group $\pi_1(S)$  of our  surfaces $S$ shall be denoted by $\Ga$, especially when we are 
seeing $S$ as a quotient of its universal covering $\tilde{S} = \HH \times \CC$. We shall also consider
more generally groups $\Ga$ acting on $ \HH \times \CC$.

\section{Surfaces isogenous to a higher elliptic product and their automorphisms}

We consider in this section a surface $S$ isogenous to a higher elliptic product,
namely, a quotient 
$$S=(C \times E)/\Delta_G,$$ where $ C $ and $E$ are smooth curves with respective genera  ${\rm genus}(C ) = : g \geq 2$ and ${\rm genus}(E)=1$, and $G$ is a finite group acting faithfully on $C $ and on $E$ such that the diagonal $\Delta_G\subset G\times G$ acts freely on $C\times E$.

\begin{remark}\label{faithful}
If $S=(C\times E)/\Delta_G$, and the action is free, we can assume without loss of generality that $G$ acts faithfully
on $C$ and on $E$, going to a minimal realization.

In fact, if the action is not faithful on $E$, the subgroup $G': = Ker (G \ra Aut(E))$ is normal and we can replace $C$ by $C/G'$:
since  $G'$ acts freely on $C$, also $C/G'$ has genus at least $2$.

 If the action is not faithful on $C$, the subgroup $G'': = Ker (G \ra Aut(C))$ acts freely on $E$, hence by translations, and we can replace $E$ by the elliptic curve $E/ G''$.
\end{remark}

Our aim is to describe explicitly the group $\Aut(S)$, in order to make precise   calculations
of the action of automorphisms on cohomology groups.

\subsection{Lifts of automorphisms to the universal covering}

As a first step  we recall in this particular case some general theory saying that, if $\tilde{S}$ is the universal covering of $S$, and $S = \tilde{S}/\Ga$, then the group $\Aut(S)$  
is the quotient group of the normalizer of $\Ga$ inside $\Aut(\tilde{S})$ by the subgroup $\Ga \cong \pi_1(S)$; see \eqref{univ-cover}.

Such a surface $S$ admits two fibrations, 
\begin{equation}\label{fibrations}
f : S \ra C/G, \ p : S \ra E/G.
\end{equation}

We have an exact sequence of fundamental groups
\begin{equation}\label{product-ex-seq} 1 \ra \pi_C \times \pi_E \ra \pi_1(S) = : \Ga \ra G \ra 1,
\end{equation}
where $\pi_C :  =  \pi_1(C) , \pi_E :  =  \pi_1(E).$

Elementary covering spaces  theory shows that each self-map $\psi : S \ra S$ lifts to the universal covering,
which is here $$q : \HH \times \CC \ra S =(\HH\times\CC)/ \Gamma.$$

As in  \cite[page 316]{topmethods}, we consider the cyclic semigroup $H$ generated by $\psi$,
or any other group or semigroup of self-maps of $S$, and we consider the semigroup  of lifts of elements of $H$, namely
$$ H' : = \{ \psi' : \HH \times \CC \ra \HH \times \CC\mid  \exists\,  \psi \in H  \text{ such that }  \psi' \circ q = q \circ \psi\}.$$

There is a short  exact sequence of semigroups
$$ 1 \ra \Ga \ra H' \ra H \ra 1.$$

 Assume now that  $\psi : S \ra S$ is a homeomorphism: then we let $H$ be the cyclic group generated by $\psi$ and 
there is a homomorphism $$ H \ra {\Aut(\Ga)}
 / \Inn (\Ga)$$ induced by conjugation by a lift $\psi'$ of $\psi$.
The fact that this action is defined only up to inner conjugation amounts to the fact that changing the base point
 $x_0$ to $y_0$  we get an isomorphism of fundamental groups $\pi_1(S, x_0) \cong \pi_1 (S, y_0)$ which is only
defined up to inner conjugation.

In particular, the group $\Aut(S)$ is  the quotient
\begin{equation}\label{univ-cover} \Aut(S) =\widetilde \sN_{\Ga} / \Ga, \end{equation}
where $\widetilde\sN_{\Ga}$ is the normalizer of $\Ga$ inside $\Aut( \HH \times \CC)$.

 Indeed any lift  $\psi'$ of an automorphism $\psi\in \Aut(S)$ yields  a commutative diagram
\[
\begin{tikzcd}
\HH\times\CC\arrow[r, "\psi ' "] \arrow[d, "q"']&  \HH\times\CC \arrow[d, "q"]\\
S \arrow[r, "\psi"]& S
\end{tikzcd}
\]
where $q$ is the quotient map. It follows that $\psi' \Gamma =\Gamma \psi '$, that is, $\psi' \in \widetilde \sN_\Gamma$. Conversely, any $\psi '\in \widetilde \sN_\Gamma$ descends to $S$.

Consider then  the exact sequence 
$$ 1 \ra \Ga \ra  \widetilde \sN_\Gamma \ra \Aut (S) \ra 1,$$
 where  any lift $\psi'$ of $\psi$ defines, via conjugation, an automorphism of $\Ga$ which is well defined up to
inner conjugation. 

Take $x_0$ to be a base point on $S$ and consider the fundamental group
$\Ga = \pi_1(S, x_0)$: then, setting $y_0 : = \psi (x_0)$, we have  $\psi_*  : \pi_1(S, x_0) \ra \pi_1(S, y_0)$,
and choosing a path $\de$ from $x_0$ to $y_0$ we get a fixed isomorphism 
$$  \pi_1(S, y_0) \cong  \pi_1(S, x_0),$$
obtained via conjugation by $\de$ (we take $\de$ to be the trivial path if $y_0= x_0$)
and,  composing $\psi_*$ with this isomorphism, we get an automorphism of $\Ga$.

By changing the lift we can make the automorphism of $\Ga$ induced by conjugation equal to 
$$\psi_* : \pi_1(S, x_0) \ra \pi_1(S, y_0)$$
(see especially \cite[page 316]{topmethods}, where $H'$ is called  the orbifold fundamental group
associated to a properly discontinuous subgroup $H$ of $\Aut(S)$: if the action of $H$ is free, $H'$ 
is the fundamental group of $S/H$,  otherwise, if $x_0$ is a fixed point of $H$, it is the semidirect product
$\Ga \rtimes H$, where conjugation is given by the action of $H$ on $\Ga = \pi_1 (S, x_0)$,
where $\psi \mapsto \psi_*$).
 
This property defines  a  lift $\tilde\psi$ of $\psi$, such that, considering the isomorphism 
 $\psi_*$  indicated above  (depending on the choice of $\de$), 
and setting, for $\ga \in \Ga$,  $\ga' : =  \psi_*(\ga)$,
$$ \tilde\psi : \HH \times \CC \ra \HH \times \CC$$
 enjoys the following important property
\begin{equation}\label{adding-prime} \ga' \circ \tilde\psi = \tilde\psi \circ \ga
\end{equation}
(observe that  both left hand side and right hand side are lifts of $\psi$ which  take the same value on
the same base point $x_0'$ lying above $x_0$).

Write now
 $$ \ \tilde\psi (t,z) = (\tilde\psi_1 (t,z), \tilde\psi_2 (t,z))=  (\tilde\psi_1 (t), \tilde\psi_2 (t,z)),$$
where the last equality follows from Liouville's Theorem,  which also implies that for $\ga \in \Ga$
(which is a lift of the identity of $S$)
$$ \ga (t,z) = (\ga_1 (t), \ga_2 (t,z)).$$

\subsection{Useful formulae and lifts to $C \times E$}
 Here we want to give another description of $\Aut(S)$ as the quotient of the Normalizer of the diagonal subgroup
$\De_G$ (inside 
the automorphism group of $C \times E$)  divided by  $\De_G$, leading to \eqref{min-product}.

Observe that  condition \eqref{adding-prime} 
$$ \ga' \circ \tilde\psi = \tilde\psi \circ \ga $$
 spells out as 

\begin{equation}\label{first-condition}  \tilde\psi_1 (\ga_1(t)) =   \ga'_1 (\tilde\psi_1 (t)) , \end{equation}
\begin{equation}\label{second-condition}   \tilde\psi_2 (\ga_1(t), \ga_2 (t,z))) = \ga'_2 ( \tilde\psi_1 (t),  \tilde\psi_2 (t,z))). \end{equation}

Conversely, any such map $\tilde\psi$ satisfying the above two equations descends to $S$.

\begin{lemma}\label{char}
The two subgroups $\pi_E, \pi_C\times \pi_E$ are  invariant under the action of $\Aut(S)$.
%\footnote{Probably they are characteristic subgroups.}
\end{lemma}

\begin{proof}
$f : S \ra C/G$ is the pluricanonical map of $S$ (Iitaka fibration of $S$), hence $f$ is preserved by automorphisms.

$f$ is a quasi-bundle, that is, all smooth fibres are isomorphic to $E$, and the product  $ C \times E$
is the normalization of the fibre product $ C \times_{C/G}S$, where the covering $C \ra C/G$
corresponds to the kernel of the monodromy map of $\pi_1(C^*) \ra \Aut (E)$, where $C^*$ is the complement of
the critical values of $f$.

Hence each  automorphism of $S$ also acts on $C$ and on $C \times E$; and since every holomorphic map of $E$ into $C$ is constant, it also acts on $E$. Therefore the subgroups $\pi_E, \pi_C\times \pi_E$ are  invariant under the action of $\Aut(S)$.
\end{proof}

Looking at elements $\ga$ in the subgroup $\pi_E < \Ga$, acting by translations on $\CC$, we obtain from \eqref{adding-prime},
since  we have $ \ga (t,z) = (t , z + \ga_2 ),$
 $$    \tilde\psi_2 (t , z + \ga_2 ) =    \tilde\psi_2 (t,z) + \ga'_2 . $$
 Hence $\partial \tilde\psi_2 (t,z)  / \partial z$ is $\pi_E$-periodic, thus constant as a function of $z$, and as usual  
 we can write 
 $$ \tilde\psi_2 (t,z) = \la z + \phi(t),$$
 and, for $t$ fixed,   $ \tilde\psi_2$ descends, as already claimed, to an automorphism of $E = \CC / \pi_E$.

A similar calculation shows  that $\ga_2 (t,z)$ is an affine function of $z$, $\ga_2 (t,z) = \la_{\ga} z + c_{\ga} (t)$
 (here $\la_\ga $ is a constant).

Looking at elements $\ga$ in  the subgroup $\pi_C < \Ga$, we see that equation \eqref{first-condition} 
says that $ \tilde\psi_1$ descends to an automorphism of $C =  \HH / \pi_C$,
which we shall denote by $\psi_1$. 

Hence by Lemma \ref{char} $\tilde\psi \in \widetilde{\sN}_{\Ga}$ descends to an automorphism $\Psi$ of $C \times E$ and we conclude that 
\begin{equation} \label{min-product}\Aut(S) = \sN_{G} / G, \end{equation}
where $\sN_{G}$ is the normalizer of $G$ inside $ \Aut( C \times E)$.

Thus  $\Psi$ is a  lift of $\psi$ to $C \times E$, and we record that
\begin{equation} \label{product-aut}\Psi(x,z) = (\psi_1 (x) , \la z + \Phi (x)),\end{equation}
 because  $\phi(\ga_1(t)) = \phi(t)$ for $\ga_1 \in \pi_C$ implies    that $\phi$ descends to 
 a holomorphic map $\Phi : C \ra E$.

Indeed,  putting, for $g \in G$,  $g' : = \Psi_* (g)  $ we have that   the conditions \eqref{first-condition} and \eqref{second-condition}
read out as:
$$  g' \psi_1 (x) = \psi_1 (g x) , \ \  g' (\la z + \Phi (x)) =  \la (g z) + \Phi (g x) .$$

Observing then  that,  for $g\in G$, $ g \mapsto g'$ is the effect of conjugating by $\Psi$, we obtain, setting $ g (z) = \e z + b$,
that  \begin{equation}\label{conj}
g' (z) = \e z + \la b - \e \Phi (\psi_1^{-1} (x)) + \Phi g  (\psi_1^{-1} (x)),
\end{equation}
hence condition \eqref{second-condition} is equivalent, setting
\begin{equation}\label{Udef}  U(x) : = \Phi (g x)  -  \e \Phi ( x) ,
\end{equation}  to:
\begin{equation}\label{U}  U(x) = U (\psi_1(x)) .
\end{equation}

This formula is particularly interesting in the following first case  (where $\e = 1\, \forall\, g$).

\subsection{The case where $G$ acts on the elliptic curve $E$ by translations,
 and $S$ is  called  a pseudo-elliptic  surface} In this case, we fix a point of $E$, called $0$, and thus  $E$ is a  group which 
acts on $S$ by translations, namely to $w \in E$ corresponds 
$$ \tau_w : C \times E \ra  C \times E , \ \tau_w (x,z) := (x, z + w)$$
 which descends to an automorphism of $S$, which we denote again  by $ \tau_w$,  and indeed we see easily that $\Aut^0(S) \cong E$.

 Moreover, since $\Ga$ centralizes $\pi_E$ under this assumption, it follows that,
 for all elements $\ga \in \Ga$,
 $$\ga_2 (t,z) =  z + c_{\ga} (t).$$

 \begin{proposition}\label{G-translates}
 Assume that $S=(C \times E)/\Delta_G$ is a pseudo-elliptic surface, that is,  $G$ acts on the elliptic curve  $E$ via translations.
 
 Then $\Aut^0(S) \cong E$ and $\Aut(S) \cong  \sN_{G} / G$, 
where $\sN_{G}$ is the normalizer of $G$ inside $ \Aut( C \times E)$.

 Moreover, letting  $Z(\De_G)$  be   the centre  of $\De_G \cong G$, and letting $\sC(G)$ be the centralizer  of $G$
 inside $\Aut(C)$, then

 $$ \Aut_{\QQ}(S) \cong (E  \times \sC(G) )/ Z(\De_G)  \Rightarrow \Aut_{\QQ}(S) /\Aut^0(S)  \cong  \sC(G) / Z(G) .$$

 \end{proposition}

 \begin{proof} 
 
 The description of $\Aut(S)$ was already given before.

Assume first that $\psi$ is numerically trivial. Then the fibration $S \ra E/G$ is preserved, hence
$\Phi(x)$ must be constant. Up to using the subgroup $E \cong \Aut^0(S)$ (the elements of $E$ have  $\la=1, \Phi = const$) we may assume that $\Phi=0$.

We are then left with automorphisms $\psi$ such that  $\Psi (x,z) = (\psi_1(x) , \la z )$.

Now, the fibration $S \ra E/G$ has   all  the fibres smooth, hence there is an exact sequence
$$1 \ra  \pi_C \ra \Ga \ra \pi_1(E/G) = : \pi_E' \ra 1.$$

Since $\pi_C$ and $\pi_E$ commute, the homomorphism of $\pi_E' \ra \Out(\pi_C)$ factors
through the surjection  $\pi_E' \ra G$ and the action of $G$ on $\pi_C$.

Hence, abelianizing,  for the first homology of $S$ we have an exact sequence
(see \cite[Proposition 2.2]{Ser91})
$$ 0 \ra \ZZ/m  \ra H_1(C, \ZZ)_G \ra H_1(S, \ZZ) \ra \pi_E' \ra 0,$$
where $H_1(C, \ZZ)_G$ is the group of coinvariants (the quotient of $H_1(C, \ZZ)$ by the subgroup generated by $\{ gv - v\}$), and $m$ is the largest integer dividing  the class  in $H^2(S, \ZZ) /H^2(S, \ZZ)_\tor$ of a fibre  of $p : S \ra E/G$.

Since $\psi$ is numerically trivial, it acts on the fibration $p$ hence on  this exact sequence and  we conclude that it acts trivially on $\pi_E' $,
 since it acts trivially on $\pi_E' \otimes \QQ$ (this
 amounts to saying  that $\la=1$; alternatively this can be deduced using the K\"unneth formula for $H^1(S,\QQ)= 
H^1(C,\QQ)^G\oplus H^1(E,\QQ)^G=
H^1(C,\QQ)^G\oplus H^1(E,\QQ)$).

We can thus  continue assuming  that $\la=1$, $\Phi=0$. Then  $\psi_1( gx) = g \psi_1( x)$,  which precisely means that 
$\psi_1 \in Aut(C)$
centralizes $G$. Hence the group of numerically trivial automorphisms is  a quotient group  of the direct product 
$ E \times \sC(G)$. The kernel of the surjection is the intersection with $\De_G$: but $ (\psi_1(x), z + u)$
is an element of $\De_G$ if and only if $u= g$ and $\psi_1 = g$, hence our assertions concerning numerically trivial automorphisms follow.
 
 \end{proof}
 
   \begin{remark}
   The previous Proposition allows to give examples of arbitrarily large quotient groups $ \Aut_{\QQ}(S)/ \Aut^0(S)$,
   as already done in \cite{CatLiu21}.
   \end{remark}
 
  \begin{remark}
 In the proof we saw   more concretely that

$$ \Aut(S) = \Aut^S (C \times E)/ \De_G,$$ 
 where
 $$  \Aut^S (C \times E) : = \{ \Psi \mid \Psi (x,z) = (\psi_1(x) , \la z + \Phi(x)), $$
 $$    \la \in \CC, \  \la \in \Aut( E, 0),  \la   G = G \subset E , g' \psi_1 (x) = \psi_1 (g x), U(x) = U (\psi_1(x)) \} $$
 where $g'$ is as in  \eqref{conj}, and $U$ is as in \eqref{Udef}.
 
 And that $\sC(G)$ is the subgroup  $\sH$  of the elements for which  $\la=1, \Phi=0$.
 \end{remark}

 \begin{theorem}\label{pseudo-elliptic}
 Assume that $S$ is a pseudo-elliptic surface, that is, $$S=(C \times E)/\Delta_G$$ is isogenous to a higher elliptic product, and $G$ acts on the elliptic curve  $E$ via translations.
 
 Then $\Aut^0(S) \cong E$ 
and the subgroup $ \Aut_{\ZZ}(S)$ coincides with $\Aut^0(S)$, except exactly when we have  pseudo-elliptic  surfaces 
 such that
 
 i)   $G= \ZZ/2m$, where  $m$ is an odd integer, and 
 
 ii) $C/G= \PP^1$ with  $ C \ra \PP^1$  branched in four points
 with local monodromies $\{ m,m, 2, -2\}$: for these surfaces we  have
 \[
 |\Aut_{\ZZ}(S)/ \Aut^0(S)| = 2
 \]
 and $ \Aut_{\ZZ}(S)$ contains  an involution permuting the first two multiple fibres. 
 \end{theorem}

\begin{proof}

\medskip
Using the notation of the previous proposition \ref{G-translates} and its proof we start with
 a preliminary observation: if we can show that  $\psi_1$ acts as the identity on 
$C/G$, then  $\psi_1$ is a lift of the identity, therefore $\psi_1$ is an element  $ g_0 \in G$. But then the product of $\Psi $ and $\tau_{g_0}$ belongs to $\De_G$, hence we conclude that $\psi \in \Aut^0(S)$.

Let us then  see  when the condition that  $\psi$  is cohomologically trivial implies that $\psi_1$  acts as the identity on $C/G$.

Since  $\psi$ acts trivially on the  cohomology groups of the quotient $C/G$, it acts  trivially
on $C/G$ if $ C/G$ has genus $ \geq 2$, and by translations on $C/G$ if $C/G$ has genus $1$.

Hence we are reduced to considering the two cases where $C/G$ has genus $1$, respectively  
where $C/G \cong \PP^1$.

In both cases we can use Lemma \ref{m=2} = Principle 3  of \cite{CatLiu21}, saying that a permutation of two branch points is only possible
if the local ($G$-valued)-monodromies are the same,  have order 2, and   moreover  all other monodromies
have odd order.

 We begin now with the hard part of the proof, which  actually requires doing calculations on the universal covering:
  because, if we have a translation $t$ on an elliptic curve $E = \CC / \Lambda$, and pick a specified  lift  of $t$  to  $\CC$, 
 a translation $ z \mapsto z + x$,
in order  to infer that $x= 0$ it is not sufficient to show that the  translation $t$ is trivial.

\subsection{Step (genus 1)}

Assume now that  $C/G$ has genus $1$: since the action of $\psi_1$ on $C/G$ is via a translation, it follows that if two branch points are permuted, then they have the same monodromies of order $2$, they differ by translation by a 2-torsion element,
and there are no more branch points than these two.

In order to proceed with our analysis, let us  describe the group $H_1(S, \ZZ)$ in general.

We have an exact sequence $ 1 \ra \pi_C \ra \pi_1^{orb}(C/G) \ra G \ra 1$, with 

{

\begin{equation}\label{orbifold}
%\begin{multline*}
 \pi_1^{orb}(C/G) = 
\left\langle  \al_1, \beta_1, \dots, \al_h, \beta_h , \ga_1, \dots, \ga_r\mid 
\prod_i [\al_i, \beta_i]  \ga_1\dots \ga_r=1, \; \ga_j^{m_j}= 1 \forall j
\right\rangle 
%\end{multline*}
\end{equation}}
 where $h$ denotes the genus of $C/G$, $\al_1, \beta_1, \dots, \al_h, \beta_h $ is a symplectic basis 
of $H_1(C/G, \ZZ)$, $\ga_j $ is the image of a geometric loop around the branch point $p_j$, and $m_j$ is the order of the local monodromy at $p_j$.
Set also for convenience $\pi^{orb} : = \pi_1^{orb}(C/G)$.

Then $ \Ga $ is a subgroup of $ \pi^{orb}\times \pi'_E$, which  has a surjection onto $ G \times G$,
and $\Ga = \pi_1(S) $ is the subgroup inverse image of $\De_G < G \times G$.

Since  $G$ is abelian, $\De_G = \Ker (G \times G \ra G)$,
via the surjection $(g_1, g_2) \mapsto g_1 - g_2$.
Hence also $\Ga= \Ker ( \pi^{orb} \times \pi'_E \ra G)$.

The surjection $\Ga \ra \pi^{orb} $ has $\mathrm{Kernel} = \{1\} \times 
\Lambda$, where $  \Lambda : = \pi_E $.

In particular $\Ga \subset \pi^{orb}  \times \pi'_E$ is generated by $ \{1\} \times 
\Lambda$, and by lifts to $\Ga$ of generators of   $\pi^{orb}$, namely of
{
$$\ga_1, \dots, \ga_r, \; \ga_{r+1} : = \al_1, \; \ga_{r+2}: = \be_1,  \dots , \ga_{r + 2h} : = \be_{h},$$}
 acting on $C \times E$ and whose further lifts to $\HH \times \CC$
have the form 
{ $$ \ga'_j (t,z) = (\ga'_{j,1}  (t), z + x_j),$$}
where the  translation $ z \mapsto z + x_j$ is in $\pi'_E$, and $m_j$ is the smallest multiple
 such that $ m_j x_j \in \pi_E$.

\begin{observation}
Then the first homology group $H_1(S,\ZZ)$, the abelianization of $\Ga$,
 is generated by the image of $  \pi_E = \Lambda $ and by the images of $\ga'_1, \dots, \ga'_{r+ 2h}$.
 \end{observation}

Moreover, setting $H^{orb} : = (\pi^{orb})^{ab} $, we get
\[
H^{orb} = H_1^{orb}(C, \ZZ) = \left( 
\bigoplus_{j=1}^h\ZZ\alpha_j  \oplus 
\bigoplus_{j=1}^h\ZZ\beta_j \oplus 
\bigoplus_{i=1}^r (\ZZ/m_i)  \ga_i  \right)/ \left\langle \sum_j \ga_j  \right\rangle.
\]
and   we obtain a homomomorphism 
$$ H_1(S,\ZZ) = \Ga^{ab} \ra H^{orb} \times \pi'_E,$$
and the exact sequence
$$ 1 \ra \Lambda \ra \Ga \ra \pi^{orb} \ra 1$$
yields an exact sequence
$$ 0 \ra \Lambda \ra H_1(S,\ZZ)   \ra H^{orb} \ra 0,$$
 because every commutator in $\Ga$ is a commutator in $\pi^{orb}$.
 
 By the same argument we have an inclusion
 $$ H_1(S,\ZZ) \subset   H^{orb} \times \pi'_E = \left(  H_1(B,\ZZ) \oplus \left( \left(\bigoplus_{i=1}^r (\ZZ / m_i) \ga_i\right)  / \left\langle \sum_j \ga_j\right\rangle \right)  \right) \times \pi'_E ,$$
because  the  commutators of $\pi^{orb} \times \pi'_E$ consist of commutators in $\pi^{orb} $.
 
 Finally, the generators $\ga'_j $ map to $(\ga_j, x_j)$.
 
  Hence we  infer that,  if {$\ga'_i  =  \ga'_j$,  
then  $x_i= x_j $.}

 The above description of $H_1(S, \ZZ)$ is equivalent to  the following presentation,  
 given by Friedman and Morgan in \cite[page 200, Lemma 7.6 and Theorem 7.7]{FM94} (but they obtained it just  in
   the special case where  the surface is obtained from a product $ D \times E$  via logarithmic transformations):
\begin{equation}\label{first-hom}
H_1(S, \ZZ) \cong \left( H_1(B,\ZZ) \oplus \Lambda \oplus (\bigoplus_{i=1}^r \ZZ \ga_i )\right) / \left\langle \sum_{j}\ga_j, m_i\ga_i- m_ix_i, 1\leq i\leq r\right\rangle,
\end{equation}
 where $\sum_i x_i =0$.  This condition  will lead in most cases to a contradiction.
\footnote{ Serrano (cf.~\cite[Proof of Theorem~4.1]{Ser91}) established a weaker result,  
the  exact sequence of abelian groups, which we shall use later on
\[
0\rightarrow H_1(E, \ZZ) \rightarrow H_1(S, \ZZ) \rightarrow H_1^{orb}(C, \ZZ) \rightarrow 0.
\]
}

In the case where $C/G$ has genus $1$,  this is indeed a contradiction since then $x_1= x_2$ and $x_1 + x_2=0$ implies that $x_1 = x_2 = 0$,
a contradiction.

\subsection{Step (genus 0)}
In the case where $C/G \cong \PP^1$, the action of $\psi$ on $C/G$ is the identity (and we are done) if there are at least 5 branch points, because by Lemma \ref{m=2}  only two branch points can be moved by $\psi$.

 Since the number of branch points is at least 4, if $\psi_1$ does not act trivially on $C/G$, then  exactly two local multiplicities are equal to $2$, and it remains to consider the case where we have exactly 
four branch points and, since the sum of the local monodromies into the abelian group $G$ is zero,
we must have monodromies $ g_1 = g_2$ of order $2$, and $ g_3 = - g_4$ of odd order $m$.

In the above notation, again we must have that, 
if  $\ga'_1 = \ga'_2$,
then 
$x_1= x_2$ and conversely, if $x_1= x_2$, then $\ga'_1 = (g_1, x_1) = (g_2, x_2) =  \ga'_2$.

Hence, eliminating $\ga'_2$, we remain with the relations 
$$ 2 \ga'_1 + \ga'_3  + \ga'_4  =0, \ 2 \ga'_1  = (0, 2 x_1) , \ m \ga'_3  = (0, m x_3) .$$
With the first we eliminate $\ga'_4 $, and remain with the two relations
$ 2 \ga'_1  = (0, 2 x_1) , \ m \ga'_3  = (0, m x_3) $. 

Since however $2 x_1 \in \Lambda$ is not divisible by $2$ in $\Lam$,
and $ m x_3 \in \Lambda$ is not divisible by any nontrivial divisor of $m$, it follows that
$H_1(S, \ZZ)$ is torsion free.

Moreover, we have that $G = \ZZ/2 \oplus \ZZ/m \cong \ZZ / 2m$.

\subsection{Existence part}

 To show that this case occurs we   construct a $G =  \ZZ / 2m$-covering $C$ of $\PP^1$
branched over four points with local monodromies $(m,m, 2, -2)$, and we consider an embedding 
of $G$ as a subgroup of  an elliptic curve $E$
(thus acting by translations).

Clearly the involution $\iota$ which exchanges the first two branch points
lifts to an involution $\s$ of $C$.

Then the corresponding $S : = (C \times E)/ \De_G$ has $q(S)=1$, $\chi(S)=0$, hence $p_g(S)=0$.

We take the automorphism $\psi$ of $S$ induced by $\Psi : = \s \times Id_E$. 

Then its action on 
$H^1 (S, \ZZ) \cong H^1 (E/G , \ZZ) $ is trivial.  

Moreover, since the Euler number of $S$ is zero,
$H^2 (S, \ZZ) $ has rank $2$, and, being torsion free, it is $ \cong \ZZ^2$. Since 
$\psi$ preserves the two fibrations, it follows that it acts trivially on the cohomology of $S$.

Since any other automorphism $\psi_1$ that will not permute the two branch points with local monodromy of order $2$ must act as the identity on $C/G$, and is therefore  the identity, as we have argued before, 
our assertion that $ |\Aut_{\ZZ}(S)/ \Aut^0(S)| = 2$ is proven.
This completes the proof of Theorem \ref{pseudo-elliptic}.

\end{proof}

\section{Surfaces isogenous to a higher elliptic product, where $G$ does not act on the elliptic curve $E$ by translations}

\label{s:surfs}

We collect here some general observations and formulae concerning  the numerically trivial and cohomologically trivial automorphisms for this class of surfaces.

A key point is that any automorphism $\psi$ of $S$ preserves the canonical elliptic fibration, and the first question is to determine
 when it acts trivially on the rational cohomology of $S$. This is easier, since the rational cohomology of $S$ is the $G$-invariant part of the cohomology of $ C \times E$.  Whereas, there is no simple commutative algebra recipe for determining
the integral cohomology of the quotient: because, commutative algebra and homological algebra
provide spectral sequences which boil down to several exact sequences of abelian groups, such as  $ 0 \ra A_1 \ra A_2 \ra A_3\ra 0$.

But if an automorphism $\psi$ acts trivially on $A_1, A_3$, it does not necessarily act trivially on the extension $A_2$.

This is the reason why in most cases we need to calculate the fundamental group $\Ga = \pi_1(S)$ and, for doing so,
we need the orbifold fundamental groups of the coverings $ C \ra B  = C/G$ and $E \ra E/G$,
and their monodromy homomorphisms onto $G$.

\medskip

Let us consider first  the group $ \Aut_{\QQ}(S)$  of numerically trivial automorphisms.

Assume that $\psi$ is numerically trivial, hence it preserves the two fibrations  of $S$  (this  is shown more generally 
 in Lemma \ref{char} to be true for each automorphism); 
hence, again, the  function $\Phi$ is a constant $c$, and we may write:
$$\Psi(x,z) = (\psi_1 (x) , \la z + c).$$
Recalling the formulae for $g(z)=   \e z + b$ we get
 $$\ \Psi_* (g) (z) = :    g' (z) = (\e z + \la b) \, \Rightarrow \, g' \psi_1 (x) = \psi_1 (g x) , \ \  \Phi (g x)  = \e  \Phi ( x) ,$$
and  we conclude from the last equality  that $ c = \e c$.

We first look at the condition of a trivial action of $\psi$ on the first cohomology group $H^1 (S, \QQ) = H^1 (C, \QQ)^G$.

This means again that $\psi_1$ acts trivially on the  cohomology group of the quotient $C/G$, hence  by Lefschetz' fixpoint formula trivially
on $C/G$ if the genus of $ C/G$, which equals  $ q(S)$, is $ \geq 2$, and by translations on $C/G$ if $C/G$ has genus $1$.

 Observe  that $G$ is a semidirect product $ T \rtimes \mu_r$, where $T$ is a subgroup of translations 
of $E$, and $\mu_r \cong \ZZ/ r$ is the group of $r$-th roots of unity ($r=2,3,4,6$), and that the surjection $ G \ra \mu_r$
yields the character $\rho$ of the representation  of $G$ on {$H^0 (E,\Omega^1_E)$.}

%$H^0 (\Omega^1_E, \CC)$.
%{
%The further condition that a non-translation $\psi$ acts trivially on the  cohomology group $H^2 (S, \QQ)$
%means, by the K\"unneth formula,  
% that the 
%cohomology group   $H^0 (C,\Omega^1_C)$,
%%$H^0 (\Omega^1_C, \CC)$, 
%as a representation of $G$,  does not contain the  character 
%dual to $\rho$.}

\medskip

The further condition that $\psi$ is cohomologically trivial, that is, it lies in $\Aut_{\ZZ}(S)$,   
requires in particular   that its action on  $\Tors (H_1(S, \ZZ))$ is trivial.

We shall see in the next section that in this latter case we may assume that $\psi_1$ is the identity, and that
the automorphism $ \psi_2 (z) = \la z + c$ is an element of the centre of $G$.

Assume from now on that  $\psi = (\psi_1, \psi_2)$,  that  $\psi_1$ is the identity and $\psi_2$ is in the centre of $G$. Then  the action of $\psi$ on the fundamental group $\Ga$ of $S$ is such that, in view of the exact sequence
\begin{equation}\label{Gamma} 1 \ra \pi_C \times \pi_E \ra \pi_1(S) = : \Ga \ra G \ra 1,
\end{equation}
 the induced action on the subgroup $\pi_C$ and on the quotient $G$ is trivial, while the
 action on $\pi_E$ is via the quotient $ G \ra \mu_r$.
 
 In particular, if $\psi_2$ is a translation, then $\psi$ acts trivially on kernel and cokernel of \eqref{Gamma}.
 
 Passing to the first homology group $H_1(S, \ZZ)$ we have an induced exact sequence
 (see for instance  \cite[Section 6.7]{topmethods}) 
\begin{equation}\label{coinvariants} H_1(C,\ZZ)_{G} \oplus H_1(E,\ZZ)_{G} \ra H_1(S, \ZZ) \ra G^{ab} \ra 0.
\end{equation}

\begin{facts}\label{facts}
%\begin{enumerate}
%\item
(1)
The left hand homomorphism in \eqref{coinvariants}
 is  injective when the group $G$ is cyclic (which is not necessarily the case here).
 
(2)
 In general the kernel of  the left hand homomorphism in \eqref{coinvariants}
 is generated by the commutators of the lifts to $\Ga$ of a system of generators of $G^{ab}$.
% \end{enumerate}
 \end{facts}

Since $\psi_1$ is assumed to be the identity, the action on $ H_1(C,\ZZ)_{G}  $ is trivial; moreover,
since $\psi_2 \in G$, the  action of $\psi_2$ on  $H_1(E,\ZZ)_{G}$ is the identity.

 \begin{remark}\label{lifts}
 Knowing that the action is trivial on the subgroup $H' < H_1(S, \ZZ)$ which is the  image of $ H_1(C,\ZZ)_{G} \oplus H_1(E,\ZZ)_{G} $ and on the cokernel $G^{ab} $ is unfortunately not sufficient to conclude that the action is trivial
on $H_1(S, \ZZ)$.

 To show triviality of the action of $\psi$ on $H_1(S, \ZZ)$ it is necessary and sufficient to show that $\psi$ fixes the lifts to
$H_1(S, \ZZ)$ of a system of generators of $G^{ab}$.

\end{remark}

Hence, in order to see whether  the action is trivial on $H_1(S, \ZZ)$ we  proceed as in the pseudo-elliptic case
 searching for  a  presentation  of $H_1(S, \ZZ)$ (see  Section \ref{2x2} for an explicit example).

 \begin{defin}
 We introduce now a new bit of notation.
 
 Write  $G = T \rtimes \mu_r$, where $T$ is  a finite group of translations of the elliptic curve $E = \CC / \Lambda$.
 
 Then there is an overlattice $\Lambda_T \supset \Lambda$ such that 
 $$ T = \Lambda_T / \Lambda.$$
 
 Moreover, the group $G$ lifts to a group of affine transformations $\sG $ of $\CC$ 
such that  $\sG / \Lambda \cong G$, hence 
$$\sG = \Lambda_T \rtimes \mu_r.$$

 We have indeed that $ \sG : = \pi_1^{orb}(E/G)$.
 \end{defin}

Recall formula \eqref{orbifold} and set:

\begin{multline*}
 \pi^{orb} : = \pi_1^{orb}(C/G) = \langle \al_1, \beta_1, \dots, \al_h, \beta_h , \ga_1, \dots, \ga_r \mid \\
\prod_i [\al_i, \beta_i]  \ga_1 \dots \ga_r=1, \ga_j ^{m_j} = 1 \, \forall j\rangle 
\end{multline*}

and  set  again  
\[
\Lambda: = \pi_E \;\;  \text{ and } \;\;   H^{orb} : = (\pi_1^{orb}(C/G))^{ab}.
\]

We have an inclusion $$ \pi_1(S) = \Ga \subset \pi_1^{orb}(C/G) \times \pi_1^{orb}(E/G) =   \pi^{orb} \times \sG, $$
 and the latter direct product  group has a surjection onto $ G \times G$,
such that  $\Ga $ is the subgroup inverse image of $\De_G < G \times G$.

The above inclusion of $\Ga$ leads then to a homomorphism 
$$H_1(S, \ZZ) \ra H^{orb} \times \sG^{ab},$$
and since we have already described the first summand, we turn to $\sG^{ab}$.

We have  the exact sequences
$$ 1 \ra \Lambda \ra \sG\ra G \ra 1 \;\; \Longrightarrow \;\;  \Lambda_G \ra  \sG^{ab} \ra G^{ab} \ra 0 ,$$
and moreover 
$$  \sG^{ab} 
\cong  (\Lambda_T)_{\mu_r}  \oplus  \mu_r  .$$

Moreover,  we have that $\mu_r$ acts on $\Lambda, \Lambda_T$ as an automorphism of
period exactly $r$. This means that they are both $\ZZ[x] / P_r(x)$-modules, where $P_r(x)$ is the $r$-th cyclotomic polynomial ($P_3(x)= x^2 + x + 1$, $P_4(x) = x^2  + 1$, $P_6(x) = x^2 - x + 1$);  and 
since $ \ZZ[x] / P_r(x)$ is a P.I.D. they are  free modules of rank 1  for $ r \geq 3$.

Whence  easy calculations (essentially the same as the ones to be performed in Lemma \ref{G}) show that 
\[
\Lambda_G \cong 
\begin{cases}
( \Lambda_T)_{\mu_2} \cong (\ZZ / 2)^2& \text{if $r=2$,}\\
( \Lambda_T)_{\mu_3} \cong \ZZ /3 & \text{if $r=3$,}\\
( \Lambda_T)_{\mu_r}= \ZZ / 2  &\text{if $ r= 4$.}\\
( \Lambda_T)_{\mu_r}= 0  &\text{if $ r= 6$.}\\

\end{cases}
\]

Hence 

\[
\sG^{ab} \cong 
\begin{cases}
(\ZZ/2)^2 \times \mu_2 & \text{for $r=2$,}\\
(\ZZ/3) \times \mu_3 &\text{for $r=3$,} \\
(\ZZ/2) \times  \mu_4 & \text{for $r=4$,}\\
 \mu_6& \text{for $r=6$.}
\end{cases}
\]

\medskip

For the  classification to come, we relate explicitly the group $G$ with its centre $Z(G)$.

\begin{lemma}\label{G}
Let $$G = T \rtimes \mu_r,$$
be a group of (affine) automorphisms of an elliptic   curve $E = \CC / \Lambda$, where the 2-generated group 
 $ T = \Lambda_T / \Lambda  $ is a group of translations  
 and $ r \in \{ 2,3,4,6\}$.
 
 Let $\e$ be a generator of $\mu_r$, and write 
 $$T \cong \ZZ/m_1 \oplus \ZZ/m_2$$
 with $m_2 \mid m_1$, more precisely 
 let  $\eta_1 , \eta_2$ be a basis  of  $\Lambda_T$,  
such that $m_1 \eta_1, m_2 \eta_2$ is a basis of $\Lambda$.

$m_1, m_2$ can be chosen  arbitrarily for $r=2$, 
while for $ r \geq 3$, writing  
$\Lambda_T = \ZZ[x] / P_r(x)$, then  there is an element $\eta$ such that $\Lambda = m_2 \eta  \Lambda_T $.

Write $ \eta=  a' + b' x \; (a', b'\in\ZZ)$. 
Then
  $ m : = m_1/ m_2 $ is unbounded, since it equals, in the  respective cases
 $r=3,4,6$,
$$ (a')^2+ (b')^2  - a'b', \ (a')^2 + (b')^2 , \ (a')^2 + (b')^2  + a'b'.$$

 We have the following cases for $Z(G)$, setting $A= \mu_r$:
\begin{enumerate}
\item If $A= \mu_2 $ then $Z(G)$ is isomorphic to a subgroup of  $E[2]\times \mu_2 $, where $E[2]$ denotes the subgroup of 2-torsion elements of  $E$.
\item If $A= \mu_3 $, then $Z(G)$ is isomorphic to a subgroup of 
$$(\ZZ /3)  \frac{ 1 - \om}{3} \times \mu_3.$$
\item If $A= \mu_4 $, then $Z(G)$ is a  subgroup of  $$ (\ZZ /2 )\frac{1}{2} (1 + i) \times \mu_4.$$
\item If $A= \mu_6 $, then $Z(G)$ is a subgroup of $A= \mu_6 $.
\end{enumerate}

Finally, either  

(i) $ Z (G)$ is contained in $T$,
the non-trivial possibilities being $(\ZZ /2)^2, \ZZ /3,  \ZZ /2$, or 

(ii) $G = Z(G)$, hence $G$ is abelian and a subgroup of
the  above groups, or

(iii)  we have the following group, that we shall call {\bf sporadic}:
\begin{equation}\label{sporadic}
 G = E[2] \rtimes \mu_4,  \ {\rm here} \ Z (G) =  (\ZZ /2 )\frac{1}{2} (1 + i) \times \mu_2.\end{equation}

\end{lemma}

With a view towards Theorem \ref{thm}, 
we point out that $|Z(G)|>4$ only occurs, in the maximal cases from (1) -- (4), when $G$ is abelian.

\begin{proof}
Let us now view  $\Lambda$ as a $\mu_r$-invariant $\ZZ$-submodule of $ \Lambda_T $.

For $r=2$, $\mu_r$ acts as multiplication by $-1$, hence we can choose $\Lambda$
as an arbitrary submodule of  $ \Lambda_T $, and  the  assertion $T \cong \ZZ/m_1 \oplus \ZZ/m_2$
 with $m_2 \mid m_1$ is a special case of the Theorem of  the Frobenius normal form for finite abelian groups
\footnote{also known as  the invariant factor decomposition.}
 (Theorem 3.9 of \cite{jacobson}).

For $ r \geq 3$, then  $\Lambda \subset \Lambda_T =   \ZZ[x] / P_r(x)$ is a principal ideal, associated to an element
$a + bx$. 

 Let $m_2$ be the divisibility index of $a + bx$ and write    
$$(a + b x) : = m_2 \eta=  m_2 (a' + b' x),$$
so that $\Lambda = m_2 \Lambda'$, where $\Lambda'$ is the principal ideal generated
by $\eta = (a' + b' x)$ (we set then $T ' = \Lambda_T  / \Lambda'$, 
so that  $ T' = T  / (\ZZ/m_2)^2$).

Clearly, a second $\ZZ$-generator of $\Lambda$  is  $ x  (m_2 \eta)$.

We have, according to $r=3,4,6$, 
$$   x (m_2 \eta) =  -b + (a-b) x, \  -b + a x, \  -b + (a+b) x.$$

Clearly $m_2 = GCD(a,b)$, while then the matrix determinant equals $m_1 m_2$, hence our assertion
concerning $m = m_1/ m_2$.

More precisely, we have that $ (b' + a'x) \eta = m x$, and then $\Lambda'$ is generated by 
$$\eta, x \eta, mx = (b' + a'x) \eta,$$
hence, if $r,s$ are integers such that $ ra' - sb'=1$,  also by  $(r  + s x) \eta, m x$ (since $(b' + a'x), (r  + s x)$
are a basis for $\Lambda_T$).

Hence, we may explicitly set 
$$ \eta_1:  = x, \eta_2 : = (r  + s x) \eta =  1 + x (s a' + b' r + s b' \de), \de= -1, 0, 1.$$
Then $$\eta_1 = x, \eta_2 = : 1 + n x$$ are a basis of $\Lambda_T$, and $ m_1 \eta_1= m m_2 \eta_1, m_2 \eta_2$
are a basis of $\Lambda$.

\medskip

We determine now  the centre   $Z(G)$ of $G=T\rtimes \mu_r$: it consists of the  elements $g = ( y, c)$ with  $ y \in \Ker (1-\e)$,
$T \subset \Ker (c-1) $. 

For $r=2$, this means that $2 y = 0$, and $c=1$ unless $m_1=2$.
This verifies case \textit{(1)}.
More precisely, it leads to the abelian case
$$
Z(G) = G = T \times \mu_2 \;\;\; \text{ if } \;\; T \subset E[2]
$$
and to $Z(G) = T \cap E[2]\subseteq (\ZZ/2)^2\subsetneq G$ in case $T\not\subset E[2]$
(where of course $G$ is not abelian).

For $r \geq 3$,  assume first that there is an element $(y,c)\in Z(G)$ such that  $c$ is a generator of $\mu_r$:
then $ (1-x) T = 0 \in T $, that is, 
$$ (1-x) \Lambda_T \subset m_2 \eta_2 \Lambda_T .$$

This immediately implies first of all that $m_2=1$, and moreover, since the ideal $I$ generated by $1-x$ has
respective index $3,2,1$  inside $\Lambda_T $, it follows that either  $ T = 0$, or  
$\eta_2 $ and $1-x$ generate the same ideal, hence we may in particular assume that $\eta_2 = 1-x$; it follows in particular that $G$ is abelian.

 The only alternative where $c$ is never a generator of $\mu_r$ and  $Z(G)\not\subset T$ is that $r=4$ while $c$ is equal  $\pm 1$.
This implies that 
$ 2 T = 0$, and then the condition $y \in \Ker (1-i)$ says that $y$ is a multiple of $ \frac{1}{2} (1+i)$.
 Hence in this case $G \subset E[2] \rtimes \mu_4$ 
 which is 
 non-abelian if and only if the inclusion is an equality, and then $ Z(G) =  (\ZZ /2 )\frac{1}{2} (1 + i) \times \mu_2$ as stated.

To conclude, 
 if $Z(G)\not\subset T$, then  we are in cases  (2) -- (4) with $G = Z(G)$, or in the sporadic case (iii) \eqref{sporadic}.

\medskip

Finally, in order to deal with the 
 case
 where $Z(G)$ is contained in the group of translations\footnote{For instance, $r=2$ and $T= (\ZZ/n)^2, n\geq 3$.},
 the condition that $ y = [v ] \in \Ker (1-\e)$
says that 
$$ Z(G) \cap T= \{ v | (1-x) v \in \Lambda \}/ \Lambda.$$ 

For simplicity of notation let us use the isomorphism  $\Lambda \cong  \ZZ[x] / P_r(x)$
 whose inverse is multiplication by $\eta$.

Then $$v \in (1-x)^{-1} \Lambda \in \Lambda \otimes \QQ,$$ 
and since 
$$ (1-x)^{-1} \; = \;\; {\rm respectively} \;\; \frac{1}{3} ( 2 + x ), \;\; \frac{1}{2} ( 1 + x ), \;\; x,$$
we see that in all three cases 

 $$ Z(G) \cap T  \;\; {\rm is \ a \ respective \ subgroup \ of } \;\;  (\ZZ/3 )\frac{1}{3} ( 2 + x ), \;\; (\ZZ/2) \frac{1}{2} ( 1 + x ),\;\;  0.$$

 Together with the analysis for $r=2$, this verifies the final statement for the case where $Z(G)\subset T$.
 \end{proof}

\begin{corollary}
Assume that $G$ is not abelian, $r \geq 3$: then   the centre $Z(G)$ consists of a nontrivial subgroup of $T$
if and only if
\begin{itemize}
\item
  $r=3$, and $ 3 | (a+b) = m_2 (a' + b')$,  (then $Z(G) = \ZZ/3$) 
 \item
  $r=4$, and $  2 | m_1 $  (then $Z(G) = \ZZ/2$).
 \end{itemize}

\end{corollary}
\begin{proof}
We have seen in the previous Lemma \ref{G} that it must be $r=3,4$.

Moreover, for $r=3$, the centre $Z(G)$ consists of multiples of $\frac{1}{3} ( 2 + x) \eta$: hence it is nontrivial
if and only if $\frac{1}{3} ( 2 + x) \eta \in \Lambda_T$, equivalently if and only if
$$ \frac{1}{3} m_2 ( 2 + x)(a' + b' x) = \frac{1}{3} m_2 ( (2a' - b') + x (a' + b')$$
has integral coefficients.

This however amounts to $ 3 | m_2$, or   $ 3 | (a' + b')$, the latter condition
 meaning that $\eta \equiv \pm 1 (1-x) \ \mod \ 3 \Lambda_T$.

For $r=4$, the condition of nontriviality of $Z(G)$ amounts instead to
$$ \frac{1}{2} ( 1 + x) \eta \in \Lambda_T \Leftrightarrow \frac{1}{2} m_2 ( (a' - b') + (a' + b') x.$$

This holds true if either $2 | m_2$, or $a' \equiv b' \ (mod \ 2)$, which is equivalent to
$2 | m$. The conclusion is that this holds if and only if $ 2 | m_1$.
\end{proof}

\subsection{Normal forms for the monodromies of $ C \ra B$ in the case $G$ abelian or sporadic}

 Lemma \ref{G}  classifies the possible groups $Z(G)$ appearing, hence gives a non sharp upper bound 
$|Aut_{\ZZ}(S) | \leq 9$ for the group of cohomologically trivial automorphisms.

In order to obtain the better bound $|Aut_{\ZZ}(S)| \leq 4$ of Theorem \ref{lem: act trivial 1} 
we need to exclude several groups $G$, and to consider all the corresponding 
surfaces $S$, which are determined by the $G$-Galois covering $ C \ra C/G= B$.

In turn these, by the Riemann existence theorem, are determined by their branch sets and by the 
isomorphism classes of their monodromy
homomorphisms.

For this reason we have to describe the  normal forms for these monodromies.

These are a finite number in the case where we want to exclude the possibility $|Aut_{\ZZ}(S) |\geq 5$, but, especially 
in the case where the genus of $B$ is at least $2$, the classes of the possible monodromies are an infinite set.

Yet we found it worthwhile to  collect a significant set of low genus (for $C$)  cases, in order to establish evidence 
for the difficulty
of finding an example with $|Aut_{\ZZ}(S) |=4$.

 We assume first that the group $G$
is as in Lemma \ref{G}, and abelian (i.e., with $G = Z(G)$, as detailed in (1) -- (4) of  Lemma \ref{G}). 

We moreover assume that the base curve $B = C/G$ has genus $h \geq 1$, and we want to find a normal form for the monodromy, that is, 
for the surjection $ Mon : \pi^{orb} \twoheadrightarrow G$, where
\begin{multline*}
 \pi^{orb} : = \langle \al_1, \beta_1, \dots, \al_h, \beta_h , \ga_1, \dots, \ga_k \mid  
\prod_i [\al_i, \beta_i]  \ga_1 \dots \ga_k=1, \; \ga_j ^{m_j} = 1 \, \forall j\rangle,
\end{multline*}
 and where we assume that $t_j : = Mon (
\ga_j)$ has order precisely $m_j$.

\begin{lemma}\label{Translations}
The case $G= \mu_r \; (r=2,3,4,6)$ cannot occur for $h=1$, and for $h \geq 2$ it leads to a trivial group $\Aut_{\ZZ}(S)$.

%In particular,  
{For $G$ abelian},  since $C$ has genus $\geq 2$,
 the number  of branch points satisfies $k\geq 2$ for $h=1$.

\end{lemma}
\begin{proof}

The first and third assertions are  clear, because if $T=0$, then the local monodromies must be trivial, hence 
there are no branch points for $ C \ra C/G=B$. 
But since we  require that $C$ has genus $\geq 2$ and {$G$ is abelian}, if $B$ has genus $h=1$ the number of branch points for $ C \ra C/G=B$ is at least 2.

For $h \geq 2$, we shall see in the proof of Theorem \ref{lem: act trivial 1} that $\Aut_{\ZZ}(S) \subset T$, hence the second assertion.
\end{proof}

In order to find a normal form for the possible monodromies, we can use
in general  the following transformations:

\begin{enumerate}
\item[a)] replace a generator $\al_i$ by $\al_i' : = \al_i \ga_j^{\pm1}$ (or similarly for $\be_i$),
\item[b)] replace $\al_i$ by $\al_i'': = \al_i \be_i^{\pm1}$ (or analogously $\be_i'' : = \be_i \al_i^{\pm1}$),
\item[c)] change the indices $i=1, \dots, h$ via any permutation.
\end{enumerate}

These operations (the second is the effect of a Dehn twist along a path $\be_i, \al_i$ or its inverse)
provide us with a different symplectic basis, and the complement of the corresponding loops is again a disk, whence we 
can complete to a new basis of the orbifold fundamental group by adding 
$\ga_1', \dots, \ga_k'$ where each $\ga_j'$ is a conjugate of $\ga_j$.

In particular, if $G$ is abelian, the local monodromies do not change.

Finally, there is a fourth operation (see  \cite[2.3 on page 250]{zimmermann}, where however the roles of $a_i, b_i$ are reversed) which, in the case where
$G$ is abelian, exchanges the monodromies $a_i : = Mon (\al_i), b_i : = Mon(\be_i)$ 
via
\begin{enumerate}
\item[d)]
$a_i \mapsto a_i, \ b_i \mapsto b_i b_{i+1} a_{i+1} , \ b_{i+1} \mapsto b_{i+1}  a_i^{-1} , \ a_{i+1} \mapsto a_{i+1}  a_i^{-1}.$
\end{enumerate}
 Note that applying b) successively allows to replace $\alpha_i$ by $\beta_i^{-1}$
(and $\beta_i$ by $\beta_i\alpha_i\beta_i^{-1}$, for instance).

\begin{lemma}\label{Mon}
 Suppose that we are in the case where $G = T \times \mu_r$ is abelian, $r < 6$,  $T \neq 0$,  and  all the local monodromies $Mon (\ga_j )=:  t_j $ are  translations of order $m_j \geq 2$.
 
 Then, after  a change of system of geometric generators  for $\pi^{orb}$,  we may assume that  $Mon (\al_1)$ is a generator of $\mu_r$ and that $Mon (\be_j) , Mon (\al_i) \in T$, $\forall j,$ and $  \forall i \geq 2$.

 If the local monodromies $t_j$ generate $T$, we can moreover assume that, $\forall j,$ and $ \forall i \geq 2$,  
 $Mon (\be_j) , Mon (\al_i) $ are trivial.

Every monodromy is then equivalent to one of
 the following normal forms for $h=1,2$.

 \begin{enumerate}
 \item[Case (I):]  
if $h=1$ or $h=2$ and $Mon (\al_2)$ and $Mon (\be_2)$ are trivial, then we have the subcases:
 \begin{enumerate}
 \item[(I-1)] $Mon (\be_1)$ is trivial,    $k \geq 2$ and $t_1, t_2$ generate $T$, or
 
 \item[(I-2)]   $k=0$, $h = 2$, and $Mon (\be_1)$ generates $T$, or 
 
 \item[(I-3)] $r=2$, $T =(\ZZ/2 )e_1 \oplus (\ZZ/2) e_2$, $k \geq 2$ and we have 
 \begin{equation}\label{st2}
 Mon (\al_1)= \e=  -1 \in \mu_2, \ Mon (\be_1)=e_1, \ t_j = e_2 \ \forall j.
 \end{equation}
 \end{enumerate}
 \item[Cases (II-III):]
    If   $h = 2$,  and we are not in case (I), then either $k=0$, or  $k\geq2$, $r=2$, $T =(\ZZ/2 )e_1 \oplus (\ZZ/2) e_2$.
     We have $ Mon (\al_1)=\e  \in \mu_2$  and either
     \begin{enumerate}
    \item[(II)]
    $k=0$, $r=2$, $T =(\ZZ/2 )e_1 \oplus (\ZZ/2) e_2$
  \begin{equation}\label{st3}
 Mon (\be_1)=e_1, \ Mon (\al_2)= e_2, Mon (\be_2)\in \{0, e_1\}; or
 \end{equation}
  \item[(II-1)]  $k=0$, $r=2$, $T =(\ZZ/2 )e_1 \oplus (\ZZ/2) e_2$, $Mon (\be_1)$  is trivial, $Mon (\al_2)=e_1, Mon (\be_2)=e_2$; or
   \item[(II-2)]   $k=0$, $r \geq 2$, $T$ is cyclic, $Mon (\be_1)$  is trivial, $Mon (\al_2)=e_1, Mon (\be_2)=0$; or
    \item[(II-3)]  $k=0$, $r \geq 2$, $T$ is cyclic, $Mon (\be_1)=e_1$, $Mon (\al_2) \in T \setminus \{0\},  Mon (\be_2)=0$; or

   \item[(III)]
    $k \geq 2$  even, $r=2$, $T =(\ZZ/2 )e_1 \oplus (\ZZ/2) e_2$.  Here $t_j=e_2$ for all $j$, $Mon(\be_1)\in \{0, e_1\},  Mon(\al_2)
   = e_1,  Mon (\be_2) =0$.
     \end{enumerate}\end{enumerate}
\end{lemma}

\begin{proof}
The composition of $Mon$ with the surjection $ G \twoheadrightarrow \mu_r$ is surjective, hence there is a global monodromy 
mapping onto a generator of $\mu_r$, since  $ r <6$ is a prime power. 
Without loss of generality, we may assume that this global monodromy is $Mon (\al_1)$ and it equals $\e$. 

With a transformation of type b), we can then achieve that $Mon (\be_1) \in T$.

For $j \geq 2$, 
 after a permutation of indices exchanging $j$ with $i+1$, 
we apply a transformation of type d) with $i=1$ to achieve that $Mon(\al_j)\in T$;
then a transformation of type b) ensures that both $Mon(\al_j), Mon(\be_j)$ have the same image in $\mu_r$
whence another transformation of type d) with $i=1$ yields
 that $Mon (\al_j), Mon ( \be_j) \in T$,
proving the first statement of the lemma
(after another transformation of type b) adjusting $Mod(\be_1)$ to be in $T$ again).
This kind of presentation will thus be fixed throughout the remainder of the proof.

If $ k>0$, we can  add to each global monodromy  any multiple of the local monodromies (transformation of type a)).

Hence, if the local monodromies generate $T$, we can obtain that all the global monodromies except $Mon (\al_1)$
are trivial.   This proves the second statement of the lemma.

Observe further that, if there are local monodromies, then $ k \geq 2$ since $\sum_j t_j=0$.
In particular, $ k \geq 2$ if $h=1$. Since we can change the ordering of the branch points, we may always assume
that the subgroup of $T$ generated by the local monodromies is generated by $t_1, t_2$.

Moreover, if $k>0$, the local monodromies generate $T$ (of prime order) unless  $r=2$, 
$T =(\ZZ/2 )e_1 \oplus (\ZZ/2) e_2$. 

\smallskip

From now on, we assume that $h\leq 2$.
Assume first that  $h=1$ or $h=2$ and $Mon (\al_2)$ and $Mon (\be_2)$ are trivial.

If $k=0$, then necessarily $h\geq 2$ by the same argument as in the proof of Lemma \ref{Translations},
hence 
we are in case (I-2).

 In the case $k \geq 2$ and $ r \geq 3$, 
since  there are local monodromies, these   generate $T$ (cyclic of prime order), hence we reach   case (I-1).

For $r=2$, there is the possibility
 that the local monodromies do not generate $T$, in which case 
we reach the normal form \eqref{st2} of case (I-3).

\smallskip

Assume now that $ h=2$, 
  but   $Mon (\al_2)$ or $Mon (\be_2)$ are nontrivial
elements of  $T$).

 We may assume, possibly 
 replacing $\al_2$ by $\be_2^{-1}$,
that  $Mon (\al_2) \in T$ is nontrivial. In particular, if $r \geq 3$, we may assume that $Mon (\al_2)$ generates $T$,
while $Mon (\be_2) =0$.

We also in general split into two cases according to

i) $Mon (\be_1)$ is trivial

ii)   $Mon (\be_1)$ is nontrivial.

Hence, for $r \geq 3$, and more generally for $T$ cyclic
($T$ is then of prime order), 
we get cases (II-2), (II-3),  since $k=0$ is automatic (else the local monodromies would generate $T$).

Look now at the   non-cyclic case $r=2$, $T =(\ZZ/2 )e_1 \oplus (\ZZ/2) e_2$
 and assume first that $k=0$; 
in the alternative  i) we have that $Mon (\al_2), Mon (\be_2) \in T$ generate $T$, hence we get case (II-1).

In the alternative  ii), we may assume $Mon (\be_1)=e_1$, and, using that the elements 
$Mon (\al_2), Mon (\be_2) \in T$, we get the normal form \eqref{st3} 
of case (II), after some transformation of type b).

There  remains to study the case $k\geq 2$.
Since the local monodromies do not generate $T$, 
we may assume that $t_j=e_2$ always, and  $Mon (\al_2), Mon (\be_j) \in \{0, e_1\}$.
Since $Mon (\al_2)$ is nontrivial by assumption, this leads to the normal form (III).
\end{proof}

\smallskip

\begin{defin}\label{def:min-mon}
Assume that we have a monodromy $Mon : \pi^{orb} \twoheadrightarrow G$ in normal form as in Lemma \ref{Mon}
with $G$ abelian (so $h=1,2$).

$Mon$ is said to be {\bf minimal} if either 
\begin{itemize}
\item  $h=2$ and we are in case (I-2), or (II), or (II-i) for $i=1,2,3$ (all with $k=0$), or in case (III) with $k=2$; or
\item $h=1$ and $k=2$, $r \in \{2,3,4\}$, and we are in  case (I-1), or $r=2$ and case (I-3); or
\item $h=1$, $r=2$, and we are  in case (I-1) with $k=3$.
\end{itemize}

$Mon$ is said to be  {\bf minimal and big} if it is minimal, $h=1$  and $|G| \geq 8$, which is equivalent to

\begin{itemize}
\item $h=1$ and $k=2$, $r \in \{3,4\}$, and we are in  case (I-1) 
\item $h=1$ and $k=2$, $r=2$ and we are in case (I-3)
\item $h=1$, $r=2$, and we are  in case (I-1) with $k=3$.
\end{itemize}

Then a new  homomorphism $Mon' : (\pi^{orb})' \ra G$ is called a {\bf simplication} of $Mon$
 if either

(i) $ h ' = h-1 = 1$, $a_2=b_2=0$ and  $k=k' \geq 2$, or

(ii) $h' = h$, $ k' = k-1 \geq 2$ and $t'_{k-1} = t_{k-1} + t_k \neq 0$, or

(iii) $h' = h$, $ k' = k-2 $,  $k' \geq 2$ if $h=1$,  and $t_{k-1} + t_k=0$,

and $Mon' = Mon$ on all  the remaining generators,  except for $\ga_{k-1}$ in case (ii).

Observe that the simplified homomorphism 
$$Mon': (\pi^{orb})' \twoheadrightarrow G,$$
is a monodromy (that is, it is surjective)  
and  in normal form, except when $r=2$, we are in case (I-1) with $T$ non-cyclic 
and $k=3$,
or in case (III) with $Mon(\be_1)=0, Mon(\al_2)=e_1$.

\end{defin}

\begin{remark}
Clearly every monodromy as in Lemma \ref{Mon} can be brought to a minimal one 
through a sequence of simplifications.
\end{remark}

\begin{lemma}\label{sporadic-mon}
Assume that we are in the sporadic case where 
$$G = ((\ZZ/2)e_1 \oplus (\ZZ/2)e_2) \rtimes \mu_4 : = T \rtimes \mu_4,$$ and all the local monodromies are translations.

Then, after  a change of system of geometric generators  for $\pi^{orb}$,  we may assume that  $Mon (\al_1)$ is the generator 
$\e= i$ of $\mu_4$  and that the other global monodromies  
 $Mon (\al_i), Mon (\be_j) \in T$, for  all $j$ and for 
 $i  \geq 2$.
 
We have the following normal forms with  $Mon (\al_i)= Mon (\be_j) =0$ for $i\geq 2, $ and $\forall j$,
where $k$ is the number of branch points:
 
 Case (III-k): $ \be_1 \mapsto 0, \ \ga_1  \mapsto e_1,  \ \ga_{2}  \mapsto e_2, \ \ga_j \mapsto d_j \in T, \ \sum_3^k d_j =  e_1 + e_2   $
 
 Case (IV-2k'): $ \be_1 \mapsto 0, \ \ga_1, \dots , \ga_{2k'}  \mapsto e_2$
 
 Case (V-(1+2k')): $ \be_1 \mapsto e_1, \ \ga_j \mapsto e_1 + e_2 $.

The above monodromies simplify to the minimal cases (IV-2) and (V-1).

 Indeed, for $ k \geq 1$, we can achieve that  $Mon (\al_i)=0$ for $i\geq 2 $,
and we have the following other normal forms.\footnote{which we do not claim to form a complete list}
 
 Case (V*(2 k')): $ \be_1 \mapsto 0,  \be_2 \mapsto e_1, \ \ga_j \mapsto  e_2 , \forall j$, $\be_j \mapsto b_j =0 \ \forall j \geq 3$,
 which simplifies to (V*(2)).
 
 Case (V**(2k')): $ \be_1 \mapsto 0,  \be_2 \mapsto e_1, \ \ga_j \mapsto  e_1 + e_2 , \forall j$, $\be_j \mapsto b_j =0 \ \forall j \geq 3$,
 which simplifies to (V**(2)).

For $k=0$ we have also the monodromies: 

Case (VI): $\al_1 \mapsto \e,  \al_2 \mapsto e_1, \ \be_2 \mapsto e_2$ and all other monodromies trivial, or

Case (VI*): $\al_1 \mapsto \e,  \be_1 \mapsto e_1 + e_2, 
\ \al_2 \mapsto e_1 $ and all other monodromies trivial, or

Case (VII): $ \al_1 \mapsto \e,  \al_2 \mapsto e_1+ e_2 , \ \al_3 \mapsto e_2$ and all other monodromies trivial,

Case (VII*): $ \al_1 \mapsto \e,  \al_2 \mapsto e_1, \ \al_3 \mapsto e_2$ and all other monodromies trivial,

Case (VIII): $ \al_1 \mapsto \e,  \ \al_2 \mapsto e_1$ and all other monodromies trivial.

Case (VI) simplifies to (VI) with $h=2$, Case (VI*) simplifies to (VI*) with $h=2$, Case (VII) simplifies to (VII) with $h=3$, Case (VII*) simplifies to (VII*) with $h=3$,
Case (VIII) simplifies to 
(VIII) with $h=2$.
\end{lemma}

\begin{proof}
    Since the local monodromies are in $T$, we may assume that $\al_1$ maps onto a generator of $\mu_4$
via the surjection $ G \ra \mu_4$; and, changing the origin in the genus one curve $E$, we may assume that 
 $Mon (\al_1) = \e$.
 
  Then, by transformations as in Lemma \ref{Mon} we may assume that all other elements $Mon (\al_i), Mon (\be_j)$
 belong to $T$: for the transformation d), even if $G$ is not abelian, yet we may apply it to the quotient group $\mu_4$
 of $G$.
 
 In the case where the local monodromies generate $T$, we may multiply the global generators by local generators,
 and obtain that $Mon (\al_i), Mon (\be_j)$ for $ i \geq2, j \geq 1$ are trivial. Then we are exactly in Case (III-k).
 
 If instead there are  local monodromies and they generate a proper subgroup $T'$ of $T$, we may assume that
 $\ga_j \mapsto e_2$ for all $j$, or  $\ga_j \mapsto e_1 + e_2$ for all $j$, depending whether  $T'$ is not $\e$-invariant
 or conversely.
 
 Then we may multiply the global generators by local generators and obtain that the other
 global  monodromies are either trivial,
 or equal to $e_1$. If they are all trivial, then we are in case (IV-k), for which we observe that $\e, \ e_2$ generate $G$,
 while $\e$ and $e_1 + e_2$ do not generate $G$.
 
 If instead they cannot be all made trivial, we may achieve 
by  moves  of type b) that all other $Mon(\al_j)$ are trivial for $ j \geq 2$.
  
If   
 $\be_1 \mapsto e_1$, and if  $ Mon (\be_j)$ are trivial, for   
 $ j \geq 2$, we are in case (V-(1 + 2k')): since $[\e, e_1]=e_1+e_2$ the sum of the local monodromies must be equal to $e_1 + e_2$,
 hence $\ga_j \mapsto e_1 + e_2$, $\forall j$,
 since we assume that $T$ is not generated by the local monodromies.
 
 If $\be_1\mapsto 0 $, we may assume $\be_2\mapsto e_1$, and if
 $ Mon (\be_j)$ are trivial, for   
 $ j \geq 3$, we are in case (V*(2k') or (V**(2k')).

 \medskip

 Now, Case (III-k) simplifies to Case (III-3), which in turn simplifies to case (IV-2), as well as case (IV-2k')
 does simplify to case (IV-2); while case (V-k) simplifies to (V-1).

 In  the case $k=0$ we observe that the commutators $[\al_j, \be_j] $ for $ j \geq 2$ map to the identity, hence
 also $[\al_1, \be_1] $  maps to the identity and 
 it follows that the monodromy of $\be_1$ is either trivial or equal to $e_1 + e_2$.

 In this latter case, we may assume that $\al_2 \mapsto e_1$, 
 and if all the other monodromies are trivial, we get (VI*).

 In the former case,  i.e.\ $\beta_1$ is trivial,
 assume that the monodromy
 restricted   to the subgroup generated by the $\al_i, \be_j$ with $i,j \geq 2$  corresponds to a surjection 
 to $T$,  hence to a surjection of the homology of a curve of genus $h-1$ to $T$, and we have several  normal forms,
 (depending on the nontriviality or conversely  of the cup product of the two corresponding first cohomology elements,
 and on the action of $\e$) leading to cases (VI), (VII), (VII*), (VIII).

 We can simplify these monodromies  by reducing the genus, until we find nontrivial monodromies, that is, until
 $h=3$ in cases (VII), (VII*), or until we reach  $h=2$ for cases (VI), (VIII).
\end{proof}

Arguing as in Lemma \ref{Mon} and Lemma \ref{sporadic-mon}, one can classify the normal forms for the monodromies in the sporadic case, but the list gets longer and longer as $h$ increases. We report here the complete list for $h=1,2$.

\begin{remark}\label{rem:list_sporadic}
Assume $G = ((\ZZ/2)e_1 \oplus (\ZZ/2)e_2) \rtimes \mu_4$, and that all local monodromies are translations, then  every monodromy is then equivalent to one of the following normal forms for $h=1,2$. 

\begin{itemize}[leftmargin=5pt]
    \item 

If $h=1$, $k\geq 1$ and we are in
 Case (III-k), or Case (IV-2k'), or  Case (V-(1+2k')). 
% These  simplify to  Case (IV-2) and  Case (V-1).

\item
If $h=2$ and $k\geq 1$, we are in
 Case (III-k), or Case (IV- 2k'), or  Case (V-(1+2k')),
 %which simplify to  Case (IV-2) and Case (V-1) for $h=1$, 
 or Case  (V*-(2k')), or  Case (V**-(2k')), or 
 
Case (V***-(1+2k')): $\al_1\mapsto \e$, $\be_1\mapsto e_1$, 
$\al_2\mapsto 0$, $\be_2\mapsto e_1$, $\ga_j \mapsto  e_1+e_2$,
which simplifies to Case (V***-1).
 
%Case (V-(1+ 2k')) and Case(V***-(1+2k')) simplify to  Case (V-1) and 
% Case (V***-1), while   Case (V*-(2k')) and 
% Case (V**-(2k')) simplify to   Case (V*-2) and Case (V**-2).
\item
 
If $h=2$ and $k=0$, we are in
Case (VI),  or Case (VI*), or Case (VIII),  or 

Case (VI**): $(\al_1,\be_1,\al_2,\be_2)\mapsto (\e, e_1+e_2, e_1, e_2)$.
\end{itemize}
\end{remark}

\begin{lemma}\label{simplify}
The effect of simplification replaces the curve $C$ (of genus $h \geq 2$) by a curve $C^0$ such that
$H_1 (C^0 , \ZZ) $ is a direct summand in $H_1 (C , \ZZ) $.

In particular, if we have a homology class $ 0 \neq c \in H_1 (C^0 , \ZZ) $ then $c$ does not map to zero  in $H_1 (C , \ZZ) $.
\end{lemma}
\begin{proof}
In case (i) $C$ is obtained from  $C^0$ by attaching $|G|$ handles, hence the assertion follows rightaway.

In cases (ii) and (iii) we have a family $C_t$ tending to $C_0$, letting the last two branch points coalesce to a single one,
possibly with trivial monodromy.

In general, $C^0$ will be the normalization of $C_0$. By the second van Kampen's Theorem
$$ \pi_1 (C_0) =  \pi_1 (C^0)* \ZZ* \dots * \ZZ \Rightarrow   H_1 (C_0, \ZZ) =  H_1 (C^0, \ZZ )\oplus \ZZ^m.$$
Whereas $C $ is homeomorphic to $C_t$ for $t \neq 0$, and we have a surjection
$$ H_1 (C_t, \ZZ) \ra   H_1 (C_0, \ZZ )$$
whose kernel is generated by the vanishing cycles, which generate a free abelian group $V$;
here $V$ is a
direct summand of $H_1 (C_t, \ZZ)$  since the quotient is torsion free.
Hence 
 $$H_1 (C , \ZZ) \cong V \oplus H_1 (C_0, \ZZ ) = V \oplus  \ZZ^m \oplus H_1 (C^0, \ZZ ).$$
\end{proof}

\subsection{Some further useful observations}
We collect here some observations  and a criterion which shall be useful in order to    achieve the full classification of
the possible actions of groups of cohomologically trivial transformations.

$\Ga$ is generated by the  subgroup $\Lambda =  \pi_E$ and by elements 
{
$$\ga'_1, \dots, \ga'_{r+2h}$$ }
acting on $\HH \times \CC$
and having  the form 
{
$$ \ga'_j (t,z) = (\ga'_{j,1}  (t),\, \e_j z + x_j),$$}
where the  affine map  $ z \mapsto \e_j  z + x_j$ is in $\sG$, and we let $m_j$ be  its smallest power
 which lies in $\pi_E$; this means, for $\e_j \neq 1$, that $m_j$ is the multiplicative order of $\e_j$,
 while, if $\e_j=1$,  that   $m_j$ is   the smallest multiple such that $ m_j x_j \in \pi_E$.

Since the action of $\De_G$ is free on $C \times E$, all local monodromies of $ C \ra B : = C/G$ 
must act as   translations on $E$: this implies  that 
all  the reductions of the   singular fibres are smooth elliptic and means that, for $i= 1, \dots, r$,  we have $\e_{2h +i} =1$,  hence we shall see that the commutators of the last $r$ generators lie in the commutator subgroup of $\pi^{orb}$. 
In particular, if $G = \mu_r$,  there are no branch points for $ C \ra B = C/G$.

 The first homology group $H_1(S,\ZZ)$, the abelianization of $\Ga$,
 is generated by the image of $\Lambda =  \pi_E$ and by the images of $\ga'_1, \dots, \ga'_{r+ 2h}$,
 which we continue to denote by the same symbol.

 A simple calculation shows that the commutators of the generators are just the commutators of the elements $\ga'_j$,
 which are of the form 
 \begin{equation}\label{comm} \left( [ \ga'_{j,1} , \ga'_{k,1} ] , z + (\e_j -1) x_k - (\e_k -1) x_j \right).
 \end{equation}
 
 Hence we see that in the  case where $\e_j = \e_k = 1$ the  commutator is  in $\pi^{orb}$.
 
 Writing as usual
  $ G = T \rtimes \mu_r$, where $ T < E$ is a finite group of translations, then the second component
  of \eqref{comm}  lies 
  in the lattice $\Lambda_T$, such that  $\Lambda_T / \Lambda = T$.
  
  Writing the elements $f\in G$  through their action on $E$,   the centralizer $\sC(G)$ of $G$ consists of  the elements  $ f(z) = \la z + c$, such that 
  $$ (\la -1)T = 0 , \ \  c \in \Ker (\e -1).$$
  
  In particular, since $\psi_2$ is in the centre of $G$, this applies to $f = \psi_2$,  and we want to see whether  the 
  automorphism  $\psi$ with first component equal to the identity and with second component  $ \psi_2$ leaves the generators $\ga'_i$ fixed.
  
  For the first component $\ga'_{i,1}$ this is obvious since $\psi_1$ is the identity, while for 
$\e_i z + x_i$ we see that conjugation by $f = \psi_2$ yields 
\begin{equation}\label{crucial} 
(\e_i z + x_i ) \mapsto (\e_i z + \la^{-1} ( x_i  + \om_i) ), \  (\e_i -1) c = : \om_i \in \Lambda,
\end{equation}
and the question is whether  $(Id,  \la^{-1} ( x_i  + \om_i) -x_i)$ lies in the commutator subgroup of $\Ga$. 

For $\la=1$, this amounts to $\om_i \in [\Ga, \Ga]$, for $c=0$, it amounts to
$(\la^{-1}  -1) x_i \in [\Ga, \Ga]$.

 In practice, in view  of the exact sequences \eqref{Gamma} and \eqref{coinvariants}, and since $\psi_1$ acts trivially on $\pi_1(C)$,
it suffices by Remark \ref{lifts} to check this for the lifts to $\Ga$ of a system of generators of  $G^{ab}$.

 We assume therefore from now on that, more generally, the elements $\ga'_i$ are just elements of $\Ga$ which
map onto a system of generators for $G$.

 From the previous discussion follows immediately the following criterion.

\begin{criterion}\label{tips}
(i)  In  the  case where  $T$ is cyclic,  then $G$ is 2-generated and the kernel of the map 
 $ H_1(C,\ZZ)_{G} \oplus H_1(E,\ZZ)_{G} \ra H_1(S, \ZZ) $ (with cokernel $G^{ab}$)  is generated by the commutator of two lifts to $\Ga$ of the two generators of $G$.
 
 (ii)
 In the  general case with $G$  abelian, then $G$  is 3-generated, and the kernel is generated by the  three commutators
 of the respective  lifts to $\pi^{orb}$ of the three generators of $G$.
 
 (iii) Assume now that there is a $j$ such that
 $ \ga'_j = (\ga'_{j,1}, \psi_2)$ and pick other elements $\ga'_i$ mapping to a system of generators of $G$. Then 
 $$  [ \ga'_j , \ga'_i ] = ([ \ga'_{j,1} , \ga'_{i,1} ], [ \psi_2 , \ga'_{i,2} ]),$$ 
  hence we  see that the  condition that  $\psi$ acts trivially on $H_1(S, \ZZ)$, which  is equivalent to
  the fact that $(Id, [ \psi_2 , \ga'_{i,2} ])$ lies in $[\Ga, \Ga]$ for all $i$,  boils down to the following property:
 \begin{eqnarray*}
 \begin{matrix}
 \;\;\;
 \text{  the commutators
  $[ \ga'_{j,1} , \ga'_{i,1} ]$, of  the respective  lifts to $\pi^{orb}$ of the}\\
  \text{  chosen system of generators of $G$, are zero in
 $ H_1(C,\ZZ)_{G}$.}
 \end{matrix}
 \end{eqnarray*}
 (iv) Finally, if $\psi_2$ is a translation, since $\psi_1$ is the identity, it suffices to check that
 $\psi$ leaves fixed the generator  $\e \in \mu_r \subset G$: hence it suffices to verify that the 
 commutator of respective  lifts to $\pi^{orb}$ of $\psi_2$ and $\e$ is zero in $ H_1(C,\ZZ)_{G}$.
\end{criterion}

This criterion will be checked in what follows by hand or with machine assistance.

\section{Properly elliptic surfaces with $\chi(\sO_S)=0$}

 As observed in Lemma \ref{lem: k1 str} minimal surfaces with Kodaira dimension $1$ 
and $\chi(S)=0$ are isogenous to a higher elliptic product, and 
case  (I) where $G$ acts by translations on the elliptic curve $E$ was treated in Theorem \ref{pseudo-elliptic};
while  case (II)  -- to be treated now -- was prepared in Section \ref{s:surfs}.

\begin{theorem}\label{lem: act trivial 1}
\label{thm2}
Let  $S$ be a minimal surface with $\kappa(S)=1$ and $\chi(\sO_S)=0$. Write $S=(C\times E)/\Delta_G$ as in Lemma~\ref{lem: k1 str}  and assume  that we are in  case (II), where  $ {\rm genus}(E/G)=0$. 
Then the following holds.
\begin{enumerate}[leftmargin=*]\setcounter{enumi}{-1}

\item
$B:=C/G$ has genus $h\geq 1$.

\item $\Aut_\ZZ(S)$ induces a trivial action on $C/G$ and on $E/G$.
\item
 $\Aut_\ZZ(S)$ is  a priori  isomorphic to a subgroup of one of the following groups:
\[
(\ZZ/2)^2 \times \mu_2,\quad \ZZ/2 \times \mu_4,\quad \ZZ/3\times \mu_3.
\]
 and it is trivial for $G = T \rtimes \mu_6$.
\item
 Indeed $\Aut_\ZZ(S)$ is isomorphic to a subgroup of the following groups
$$(\ZZ/2)^2 , \ZZ/3,$$
and more precisely:
\item  { 
 If $h \geq 2$,
or if $G$ is not abelian, then  $\Aut_\ZZ(S)$ is  a subgroup of one of the following groups
 \[
(\ZZ/2)^2, \quad \ZZ/3,
\]
And the cases $\Aut_\ZZ(S) = \ZZ/2, \ZZ/3$ do in fact occur.

More precisely, if $h \geq 2$, or if $G$ is not abelian and not sporadic ( in the sense of Lemma \ref{G}) , 
 then $\Aut_\ZZ(S)$ acts on $E$ as a group of translations.}
 \item
   When   $h=1$,   $\Aut_\ZZ(S)$  is trivial in the cases $$ G  \cong (\ZZ/2)^2 \times \mu_2,\  \ \ZZ/3\times \mu_3, \ \ZZ/2 \times \mu_2.\ $$
 If instead  $h=1$, $G \cong \ZZ/2 \times \mu_4,$ $\Aut_\ZZ(S)$ is either trivial or $= \ZZ/2$, and the latter case does
effectively occcur.
\item We have  $\max |\Aut_\ZZ(S)|\in \{3,4\}$.
\end{enumerate} 
\end{theorem}

\begin{proof}

(0) As indicated before, this follows from Lemma \ref{lem: k1 str} (2).

(1) Note that $\Aut_\ZZ(S)$ preserves any fibration structure.

Let $f\colon S\rightarrow B = C/G$ be the fibration induced by the first projection of $C\times E$. Since ${\rm genus}(E/G)=0$, we have $q_f=0$, and hence the cohomology classes of (the reduced divisors associated to) multiple fibres are different from one another by Lemma \ref{lem: mult fib qf=0}. Therefore, $\Aut_\ZZ(S)$ preserves each multiple fibre of $f$. Since the multiple fibres are exactly $f^*b$ for the branch points $b$ of the quotient map $\pi\colon C\rightarrow C/G$, the induced action of $\Aut_\ZZ(S)$ on $C/G$ fixes each branch point of $\pi$.

We claim that  $\Aut_\ZZ(S)$ acts trivially on $C/G$: in fact, $\Aut_\ZZ(S)$ acts trivially on $H^1(C/G, \CC)\hookrightarrow H^1(S,\CC)$, hence this is clear for  $h \geq 2$ by the Lefschetz formula.
 If $h= 1$ then $\Aut_\ZZ(S)$ acts by translations on $C/G$, and since moreover  $\Aut_\ZZ(S)$ fixes pointwise the nonempty set of branch points of $C\rightarrow C/G$, it must act trivially on  $C/G$. 

Next, we show that $\Aut_\ZZ(S)$ acts trivially on $E/G$; indeed, 
by the Riemann--Hurwitz formula, the branch multiplicities  of $E\rightarrow E/G \cong \PP^1$ cannot be of the form $(2,2,m_1,\dots, m_r)$, where the $m_i$ are odd numbers (since $\sum_j (1- \frac{1}{m_j})=1$). It follows from Lemma \ref{m=2}  that $\Aut_\ZZ(S)$ preserves each multiple fibre of the induced fibration 
$$ p \colon S\rightarrow E/G,
$$ and hence the induced action of $\Aut_\ZZ(S)$ on $E/G$ fixes each branch point of $E\rightarrow E/G$. By the Riemann--Hurwitz formula again, there are at least three branch points of  $E\rightarrow E/G$, which are then fixed by the induced action of $\Aut_\ZZ(S)$. It follows that $\Aut_\ZZ(S)$ acts trivially on $E/G$.

\medskip

(2) Since $\Aut_\ZZ(S)$  acts trivially on $C/G$ and on $E/G$, there is an inclusion
\[
\Aut_\ZZ(S) \subset N_{G\times G}(\Delta_G)/\Delta_G \cong Z(G),
\]
where $N_{G\times G}(\Delta_G)$ denotes the normalizer of $\Delta_G$ in $G
\times G$, $Z(G)$ denotes the centre of $G$, and the isomorphism $N_{G\times G}(\Delta_G)/\Delta_G \cong Z(G)$ is given by $(\sigma_1,\sigma_2)\mapsto \sigma_1\sigma_2^{-1}$. Since by assumption $G$ acts faithfully on $E$ and $g(E/G)=0$, we have
$$G=T\rtimes A,
$$
where $T$ consists of translations and $A$ is a nontrivial group   $A \cong \mu_r \cong \ZZ/r$ with $r\in\{2,3,4,6\}$
as studied in Lemma \ref{G}.

Since $\Aut_\ZZ(S)|_{C/G}$ is trivial,  we also have an induced action of $\Aut_\ZZ(S)$ on the fibre $E$ of $f\colon S\rightarrow B$.

Since $\Aut_\ZZ(S)$ can be identified with a subgroup of $Z(G)$, it is isomorphic to a subgroup of one of the following groups (here the $\mu_r$ appearing means that $G = T \rtimes \mu_r$
as in  Lemma \ref{G}):
\[
(\ZZ/2)^2 \times \mu_2,\quad  \ZZ/3 \times \mu_3, \quad  \ZZ/2 \times \mu_4,  \quad \mu_6.
\]

 However, if $\Aut_\ZZ(S) \subset \mu_6$, and it is not trivial, 
 then by Lemma \ref{G} $G = Z(G)$, hence $G = \mu_6$, which is excluded {
by Lemma \ref{Translations} if $h=1$, while if $h\geq 2$ it will be excluded in point (4).} 

This completes the proof of the second assertion of the proposition.

\medskip

(3) follows clearly from (4) and (5).

(4): since $\psi_1$ induces the identity on $C/G$,
  and as observed we may assume that $\psi_1$ is the identity,  we claim that the condition that $\psi$  is numerically trivial 
 is equivalent to the condition that the representations of $G$ corresponding to   the inclusion  character $ A = \mu_r \ra \CC^*$ and to  its complex conjugate do not occur in the representation of $G$ on $H^1 ( C, \CC) $.
 
 In fact, the cohomology of the quotient is the $G$-invariant cohomology of the 
 product $C \times E$, which, by the K\"unneth formula, is the tensor product of the respective cohomologies of $C, E$.
 
For the first cohomology, we have  $ H^1 ( S, \CC) = H^1 ( C, \CC)^G = H^1 ( C/G, \CC)$, because  $H^1 ( E, \CC)^G = 0$, and since we know that $\psi_1$ acts as the identity on $C/G$, the action on this first cohomology group is trivial.

For the second cohomology
 \begin{multline*}
 H^2 ( S, \CC) = \left(H^2(C, \CC)\otimes H^0 ( E, \CC) \right)^G \oplus \left(H^1 ( C, \CC) \otimes H^1 ( E, \CC) \right)^G \\
 \oplus \left(H^0(C, \CC)\otimes H^2 ( E, \CC)\right)^G,
 \end{multline*}
  and the action is trivial on the first and third summand.
 
Write here $H^1 ( C, \CC) = \oplus_j V_j$
 as a sum of irreducible representations. 
 
 We have  $H^1 ( E, \CC) = W \oplus \bar{W}$,
 where the 1-dimensional representation $W = H^0 (E,\Omega^1_E)$ corresponds to  the inclusion  character $ A = \mu_r \ra \CC^*$. 

 Set $W_1 : = W, \ W_2 : = \bar{W}$.
  
 By Schur's lemma,  if $V_j,  W_i$ are irreducible representations, then $ (V_j  \otimes W_i)^G=0$
 unless $ V_j  \cong W_i^{\vee} \cong \bar{W_i}$.
 
 Since $\psi_1$ is the identity on $C$, it acts as the identity on $H^1 ( C, \CC) $, hence on each summand
 $V_j$. Hence, if there is a $j$ such that $V_j \cong \bar{W_i}$, 
 and  $\psi_2$ is not a  translation, then 
 the action on second cohomology  is nontrivial.
 
 \medskip
 Hence, either 
 
 \begin{enumerate}
 \item[(i)] 
 $\psi_2$ is a translation for all 
 $\psi\in\Aut_\ZZ(S)$,  or
 \item[(ii)]  
 there is  a  $\psi\in\Aut_\ZZ(S)$  such that $\psi_2$ is not a translation.
 \end{enumerate}

 In case (i), Lemma \ref{G} leads exactly to the possible groups $Z(G)\cap T$ stated in (4).
 
 {It is important to observe that, again by Lemma \ref{G}, if $G$ is not abelian and not sporadic, then $Z(G) \subset T$,
 hence we are in case (i).} 
 
 In case (ii), we moreover have that 
 no $V_j$ can be isomorphic to $W_1$, nor $W_2$, which is equivalent 
 to  requiring is that $$\left(H^1 ( C, \CC) \otimes H^1 ( E, \CC) \right)^G=0,$$ and in turn this is equivalent 
 to requiring $b_2(S) = 2$.

 In turn, the condition $b_2(S) = 2$, since $q(S) $ equals the genus of $B = C/G$, implies that $e(S) = 4 - 4 h$.
 
 Since we know that $e(S) = 0 = e ( C \times E)$, the conclusion is that
 in case (ii)  it must be $h=1$,
 that is, $B$ is an elliptic curve. 

{
 In Sections \ref{main-ex} and \ref{main-ex-rel}
 we exhibit examples having $h=2$ and realizing $\Aut_\ZZ(S)=\ZZ/3, \ZZ/2$.
 This completes the proof of (4).
}

 \medskip
 
 (5) 
 By what we have seen so far, it suffices to concentrate on case (ii):  
 here $B$ has genus $1$, the local monodromies of $ C \ra C/G = B$
are just translations (observe that  the branch locus of  $ C \ra C/G = B$ is nonempty  because we want that $C$ has genus $\geq 2$),
and our automorphism is induced by $ (\psi_1, \psi_2)$,
where $\psi_1$ is the identity,  $G=Z(G)$ and $ \psi_2 \in Z(G)$ 
is not a translation.
Moreover, 
 $q(S)=1, p_g(S)=0$,
and $ b_2(S) = 2$. Hence 
 the two fibre images $E, C$
of the horizontal and vertical curves 
in $C \times E$ span a lattice $\sL$ of rank two, on which the group of automorphisms of $S$  acts trivially.

The next lemma shows that, in case where $T$ is cyclic, we may just focus on showing that the action of $\psi$ 
is  the identity
on $ H_1(S, \ZZ)$.

\begin{lemma}\label{Tcyclic}
Assume that $G = T \rtimes \mu_r$, 
where $T$ is cyclic, that $B$ has genus $1$, that  the local monodromies of $ C \ra C/G = B$ are translations, and there is one local monodromy generating $T$. Then, if an automorphism $\psi$,
 induced by $ (\psi_1, \psi_2)$,
where $\psi_1$ is the identity and $ \psi_2 \in Z(G)$,  acts as the identity
on $ H_1(S, \ZZ)$, it also acts  as the identity on  $ H^2(S, \ZZ)$.
\end{lemma} 
\begin{proof}
Assuming  that $\psi_1$ is the identity 
    implies that 
   the submultiple fibres (the reduced multiple fibres) of the fibration $f :  S \ra B$ are left fixed by $\psi$.
   
Then $\psi$ acts as the identity on the  subgroup
of $H^2(S, \ZZ)$ generated by these submultiple fibres  of $f$ and by  the smooth  fibres $E,C$
of the two fibrations $f,p$ (see \eqref{fibrations}).
We can moreover use Lemma \ref{m=2}
to infer that the submultiple fibres of the fibration $S \ra E/G$ are left fixed,
since the  branching indices are, by Hurwitz' formula,  never of the form  $(2,2, m_3, \dots )$ where $m_j$ is odd for $ j \geq 3$.

It is not true in general  (see Section \ref{2-2}) that these submultiple fibres generate the torsion subgroup, but by our assumption  
 the action on the torsion subgroup of  $H^2(S, \ZZ)$ is the identity.
 
It suffices to exhibit two divisor classes with self intersection zero which are left fixed by $Z(G)$, and such that their intersection
product equals $1$: because then their image in  the quotient group $H^2(S, \ZZ)/ H^2(S, \ZZ)_\tor$ yields  a basis,
hence the  two divisor classes and the torsion classes yield a system of generators for $H^2(S, \ZZ)$ which are left fixed by $Z(G)$.

These two divisor classes are found as follows.

By assumption, there is a fibre of $f$ which is multiple of order $|T|$, which we write as $|T| M_1$.

Now, the fibration $ p : S \ra E/G$ is induced by  $C \times E \ra E \ra E/G$, and we observe
 that there is a fixed point $y$ in $E$ for
$\mu_r$, hence there is a multiple fibre of $p$ occurring  with multiplicity $r$, that we write  as $ r M_2$ (with $M_2$  isomorphic to 
$C / \mu_r$).

Now, denoting by $\pi : C \times E \ra S$ the projection, we have:
\[
M_1 \cdot M_2 = \frac{1} {|T| r } E \cdot C = \frac{1}{|G|} E \cdot C
\]
and
\[
|G| E \cdot C = \pi^* (E) \cdot \pi^* (C) = |G|^2
\Rightarrow M_1 \cdot M_2 = 1 .
\]

Hence we have  proven  that $\psi$ acts as the identity 
 on a system of generators of $H^2(S, \ZZ)$, so it acts as the identity on 
 the full cohomology group.
\end{proof}

We conclude now the proof of step (5) using  the fact that the monodromies of  the covering $C \ra B$ in these cases
 simplify to the
four   minimal big cases, or to the case of Proposition \ref{2x2}; we shall use  this proposition and   the computer assisted result Theorem \ref{Magma}  which is shown in the appendix:
 implying that in the  four   minimal big cases with $G = Z(G)$ of  order $\geq 8$,
 and in the case of $h=1$, and $G = \ZZ/2 \times \mu_2$, the group   $\Aut_\ZZ(S)$ is trivial except
 for  case (I-1) of $ G= \ZZ/2 \times \mu_4$, where  $\Aut_\ZZ(S) \cong \ZZ/2$.

 The condition that $\Aut_\ZZ(S)$ is trivial is equivalent, in view of Criterion  \ref{tips},   to the fact that,
  for any  element  $g \in Z(G)$,  the commutators of the lifts to $\pi^{orb}$
 of a minimal system of  generators of $G$ and of  $g$ are  all  nontrivial in $H_1(C, \ZZ)_G$. 
  
  Now,   for the   four    minimal big cases, with respective groups

 \[
 \quad  \ZZ/3 \times \mu_3, \quad  \ZZ/2 \times \mu_4,
\quad (\ZZ/2)^2 \times \mu_2, \quad (\ZZ/2)^2 \times \mu_2,
\]
 the monodromy is  in normal form and then 
 the lifts of the generators are respectively 
 $$ \ (\al_1, \ga_1), \
 (\al_1, \ga_1), \ (\al_1, \be_1, \ga_1), \ (\al_1, \ga_1, \ga_2).$$

 An immediate extension of Lemma \ref{simplify} shows that under the simplification process 
 $H_1(C, \ZZ) = H_1(C^0, \ZZ) \oplus W$ is a $G$-equivariant splitting, therefore we have
 equality $H_1(C, \ZZ)_G = H_1(C^0, \ZZ)_G \oplus W_G$.
 
 Therefore, if a commutator is not zero in $H_1(C^0, \ZZ)_G $, then it is also nonzero inside
 $H_1(C, \ZZ)_G $.
 
 Because by simplification every monodromy can be carried to a minimal one, 
 statement (5) follows now from Theorem \ref{Magma} 
 and Proposition \ref{2x2}.

 \medskip

(6) The final statement follows by combining (3), (4) and (5).
 \end{proof}
 
 \medskip
 
 \begin{remark}
 Concerning  the case where the genus $h$ of $B$ is $ \geq 2$,
  or  more generally where we have a subgroup $H$ of $Z(G)$ consisting only of translations:
  
  (i) If $h \geq 2$ we can prove that the subgroups in the list (3) of Theorem \ref{lem: act trivial 1} are exactly the groups of numerically trivial automorphisms.
  
  (ii) If we have a subgroup $H$ of $Z(G)$ consisting only of translations, then $H$ acts trivially on $\pi^{orb} \times \Lam_T$, and we must see what happens for $\mu_r$, using Criterion \ref{tips}.  
 \end{remark}

%\begin{question}

%Are there automorphisms in $\Aut_\ZZ(S)$ induced by $\psi_1 = Identity$ and  $\psi_2$ in $Z(G)$
%a nontrivial  translation? 
%\end{question}

\subsection{Proof of Theorem \ref{thm}:}
Theorem \ref{thm} follows immediately from Theorem \ref{pseudo-elliptic} and Theorem \ref{thm2}.

\subsection{The main example}\label{main-ex}

We show  here that  the case where $G=\ZZ_3\times \mu_3$, $h=2, k=0$, and the monodromies are of type (I-2) 
leads to $ \Aut_\ZZ(S) \cong \ZZ/3$.

Here, if  $t$ is a generator of $T$, the monodromies are:
$$ \al_1 \mapsto \e,  \be_1 \mapsto t , \  Mon (\al_2) = Mon (\be_2) \ {\rm is \
 trivial}.$$
 
 In this case we have that the orbifold fundamental group $\pi'$ on the side of the elliptic curve $E$
 is such that $$ \pi' \supset \pi_1(E) = \Lambda = \ZZ e_1 \oplus \ZZ \om,$$
$\pi'$ is generated by $e_1 : = 1 \in \CC$, by $\e \in \mu_3$ and by $\tau : = \frac{1-\om}{3}$.

Here $t$ is the class of $\tau$ modulo $\Lambda$, $\om = e_1 - 3 \tau$, and 
conjugation by $\e$ sends
$$ e_1 \mapsto e_1^{\e} = \om = e_1 - 3 \tau, \ \ \tau  \mapsto \tau^{\e} = \frac{\om - \om^2}{3} = 
 \frac{2 \om +e_1}{3} = e_1 - 2 \tau.$$ 
 
 Since $S \ra B$ is a fibre bundle with fibre $E$, we have an exact sequence
 $$ 1 \ra \Lambda \ra \Ga \ra \pi_1(B) \ra 1,$$
 hence $\Ga \subset \pi_1(B) \times \pi'$ is generated by $\Lambda = \ZZ e_1 \oplus \ZZ \om$
 and by lifts of the generators of $\pi_1(B)$, namely
 
 $$ \tilde{\al_1} : = (\al_1, \e) , \ \tilde{\be_1} : = (\be_1, \tau) ,\ \al_2 = (\al_2, 0), \be_2 = (\be_2, 0).$$
 
Clearly $\al_2, \be_2, \tilde{\be_1}$ centralize $\Lambda$, while

$$  e_1^{\tilde{\al_1}} = e_1^{\e} = \om ,  \ \om ^{\tilde{\al_1}} =  \om^2 = -e_1 - \om.$$ 

Finally, 
$$ [ \tilde{\al_1}, \tilde{\be_1}] [\al_2, \be_2] = (1, [\e, \tau] )= (1, \e \tau \e^{-1} \tau^{-1} )= (1, e_1 - 3 \tau) = (1, \om).$$

Because of the relations $\om \sim  0, e_1  \sim \om, \om \sim -e_1 - \om$,
the image of $\Lambda$ inside $H_1(S, \ZZ)$ is zero, hence $H_1(S, \ZZ) \cong H_1(B, \ZZ) = \ZZ^4$
is torsion free and we conclude that $\Aut_\QQ(S)= \Aut_\ZZ(S)$.

Then since $h=2$, $\Aut_\QQ(S)= T$, hence $ \Aut_\ZZ(S) = \ZZ/3$.

\subsection{A relative of the main example, with $ \Aut_\ZZ(S) = \ZZ/2$}\label{main-ex-rel}
We assume now that $h=2$, $k=0$, $T \cong (\ZZ/ 2) t$, $r= 4$ and we have a monodromy of type (I-2), i.e.
$$ \al_1 \mapsto \e,  \be_1 \mapsto t , \  Mon (\al_2) = Mon (\be_2) \ {\rm is \
 trivial}.$$
 
 Here $\pi'$ is generated by $e_1 : = 1 \in \CC$, by $\e \in \mu_4$ and by $\tau : = \frac{e_1+ e_2}{2}$,
 where $e_2 = i$ since $r=4$.
 
 Again we have a fibre bundle and if we show that the image of $\Lambda$ is zero in $H_1(S, \ZZ)$, then again there is no torsion and $ \Aut_\ZZ(S) = \ZZ/2$.
 
 Again 
 $$ [ \tilde{\al_1}, \tilde{\be_1}] [\al_2, \be_2] = (1, [\e, \tau] )= (1, - e_1).  $$
 
Hence  we get $ e_1 \sim e_2,  e_1 =0$ and we are done.

\subsection{Examples using   the Reidemeister Schreier method}
\label{2-2}

In this section we analyse whether, 
{in the case $h=1$ and  $G=\ZZ/2\times \mu_2$, the automorphisms in $Z(G)$
are cohomologically trivial. }
%in  all  the examples having the action on $E$ as contemplated in 
%the proof of (2) of  Theorem
%\ref{lem: act trivial 1},  the automorphisms in $Z(G)$
%that we have described are cohomologically trivial. 
 
 In order to do so, we need to use the Reidemeister-Schreier method, which, given a group presentation (here: for  $\pi^{orb}$)
with generators and relations, allows to find a presentation for a subgroup (here: $\pi_1(C)$), (see \cite{mks}, pages 86-95, \cite{l-s} pages 102-104)

\begin{proposition}\label{2x2}
 Assume that  $ G \cong \ZZ/2 \times \mu_2$: then  $\Aut_\ZZ(S)$ is trivial 
\begin{itemize}
\item
 for the  $G$-covering $C$ of an elliptic curve $B$  with minimal monodromy of type (I-1),
 and 
 \item
 more generally if $h=1$.
 \end{itemize}

\end{proposition}

\begin{proof}
 Here we have two branch points  and the orbifold fundamental group is generated 
by $ \al, \be, \ga_1$ subject to the relations 
$$\ga_1^2= ([ \al, \be] \ga_1)^2 = 1$$

 To study the action on  $H_1(S, \ZZ)$ we use  a presentation of
$\Ga$ obtained using the Reidemeister-Schreier method.

Let 
$$\xi : = (\al, \e),\, \eta : = (\ga_1, \eta_1) \Rightarrow \xi^2 = \al^2 \in \pi_1(C),\, \eta^2 = 2 \eta_1 \in \Lambda.$$

Recall the exact sequence
$$ H_1(C, \ZZ)_G \times \Lambda_G \ra H_1(S, \ZZ) \ra G = (\ZZ/2)^2 \ra 0 $$
where the kernel of the first homomorphism to the left is generated by the commutator
$$\tau: =  [ (\al, \e),(\ga_1, \eta_1)] = ( \al \ga_1 \al^{-1} \ga_1, - 2 \eta_1).$$

As we already calculated, $\Lambda_G = \Lambda / 2 \Lambda =  (\ZZ/2) 2 \eta_1  \oplus (\ZZ/2)\eta_2$.

$C$ has genus $3$, and the Reidemeister-Schreier method shows that $\pi_1(C)$  is generated by
$$ \be,\, \al \be \al^{-1} ,\, \ga_1 \be \ga_1,\,  \al \ga_1 \be \ga_1 \al^{-1},\, \al^2,\, \ga_1 \al^2 \ga_1,\, \al \ga_1 \al \ga_1 .$$

Denote the above sequence of generators by $b_1, b_2, b_3, b_4, x_5, x_6, y_7$: these are too many, but it is immediate
to verify that
$$ b_2 b_1^{-1} x_6 y_7^{-1} b_4 y_7 x_6^{-1}  b_3^{-1}  =( \al \be \al ^{-1}  \be ^{-1} \ga_1)^2 = 1.$$

Hence any one of the $b_i$'s can be eliminated from the system of generators and we are left with exactly six generators, whose image in
$H_1(C, \ZZ)$ will be a free $\ZZ$-basis.

Denoting their image in $H_1(C, \ZZ)$ by the same symbol  and writing $a \sim b$ if two elements have the same image in 
$H_1(C, \ZZ)_G$, we see  that, conjugating by  $\xi, \eta$ the first six generators, we get exactly the relations
$$ \be \sim \al \be \al^{-1} \sim \ga_1 \be \ga_1 \sim \al \ga_1 \be \ga_1 \al^{-1},  \al^2 \sim \ga_1 \al^2 \ga_1$$
and we denote by $b$ the image of the first four  generators, by $x$ the image of the fifth and sixth,
and by $y$ the image of $ \al \ga_1 \al \ga_1$.

 Moreover, conjugating the last generator by $^{-1} \eta$ we get exactly  the relations
$$ \al \ga_1 \al \ga_1 \sim \ga_1 \al \ga_1 \al = \ga_1 \al^2 \ga_1 ( \al \ga_1 \al \ga_1)^{-1} \al^2 ,$$ 
hence $y = 2x - y$, so that $ 2 (x-y)=0$. 

We also observe that $ \al \ga_1 \al^{-1} \ga_1 =  \al \ga_1 \al \ga_1 \ga_1 \al^{-2} \ga_1$,
hence the image of $\tau$ equals $ (y-x, - 2 \eta_1) = ( y-x, 0)  - 2 \eta$.

Putting everything together, we get that $H_1(S, \ZZ)$ is generated by
$$ \xi,\, b,\, \eta,\, \eta_2,\, y,$$
with relations   (since $ x = 2 \xi$):

$$ 4 \xi =  2y,\, (2 \xi -y)  =  - 2 \eta , \, 4 \eta=0,\, 2 \eta_2 =0,$$
where however  the first relation is implied by the others.

Hence, eliminating $y$, we get:
 $$H_1(S, \ZZ) = \ZZ b\oplus 
 \ZZ \xi \oplus (\ZZ/2) \eta_2 \oplus (\ZZ/4) \eta.$$
 This group is larger than $H^{orb}$, and the action of $\psi$, if $\psi_2 = -1$,
  equals  the identity on $\xi, b, \eta_2$,
 but not on $\eta$, which is sent to $-\eta \neq \eta$.

 Similarly if $\psi_2 (z) = z + \eta_1$, see the calculations of Criterion \ref{tips}, the action is not the identity since
 $\om = 2   \eta_1 = 2 \eta \neq 0$ inside $H_1(S, \ZZ)$, 
 hence the  translation by $\eta_1$ sends $\xi$ to $\xi + 2 \eta$, and leaves $\eta$ fixed.
 
 Hence, not only these two elements do not belong to $\Aut_{\ZZ}(S)$,
but also  their product $g$ does not, since it does not keep fixed $\xi$ and $\eta$.

 For the second assertion, it suffices to observe that for $h=1$ the monodromy simplifies to
  a minimal monodromy of type (I-1).
\end{proof}

\begin{remark}\label{3,3}
One can ask whether similarly   this is not the case  also for  $ G \cong \ZZ/3 \times \mu_3$,
and for the  $G$-covering $C$ of an elliptic curve $B$ with minimal monodromy of type (I-1).

We have  here $r=3$, $ G = T \times \mu_3$, where $T$ is generated by $\eta_1 = \frac{1-\e}{3}$.
 
 $T$ being cyclic, we take  only two branch points, and we take the following two elements of $\Ga$
 mapping onto a set of generators for $\De_G$:
 $$ (\al, \e), (\ga_1, \eta_1).$$
 
We  use now the criterion given in Criterion \ref{tips}.
 
 We have that $\psi (x,z) = (x, \la z + c)$, where $\la^3 =1$ and $c$ is a multiple of $\eta_1$.
 
 Hence the  generators are left invariant by $\psi$ if and only if
 $$ \la^{-1} (\e-1) c, (1- \la^{-1}) \eta_1$$ 
 are in the commutator subgroup of $\Ga$: for this a sufficient condition would be 
that  they map to zero in $\Lam_G$.
 
 This holds for all $\la, c$ if and only if $ (\e-1) \eta_1 = 0 \in [\Ga, \Ga]$.
 
 Since  $ (\e-1) \eta_1 = - \frac{1}{3} (\e-1)^2 = \e$, which is the second period $e_2$ of $\Lam$,
 we conclude that we do not have zero in $\Lam_G$ because
 $$ \Lam_G  = (\Lam)_{\mu_3} = \langle 1, \e  \rangle /  \langle \e -1, \e- \e^2  \rangle \cong (\ZZ/3) \e.$$
 
However, in calculating the image of $ \Lam_G $ inside $\Ga^{ab}$, we must divide by the subgroup
generated by the commutator 
$$ [(\al, \e), (\ga_1, \eta_1)] = (\al \ga_1 \al^{-1} \ga_1^{-1} , (\e-1) \eta_1) =   (\al \ga_1 \al^{-1} \ga_1^{-1} , \e),$$
 hence $(0, \e)$ is  in  $[\Ga, \Ga]$ if and only if  $\al \ga_1 \al^{-1} \ga_1^{-1}$ is
 trivial in $H_1(C, \ZZ)_G$, as predicted by Criterion \ref{tips}.

  The answer is again negative, as shown in the appendix, and is found again via the Reidemeister Schreier method. 
 Since  the calculations here are quite complicated, it was necessary to use  MAGMA   for computer calculations (see Theorem \ref{Magma}).

\end{remark}

 \subsection{The group of homotopically trivial automorphisms of properly elliptic surfaces with $\chi(\sO_S)=0$ 
 is  just $\Aut^0(S)$.}

\begin{prop}\label{homotopic}
Let  $S$ be a smooth projective surface with $\kappa(S)=1$ and $\chi(\sO_S)=0$. Then $\Aut_\#(S)=\Aut^0(S)$.
\end{prop}
\begin{proof}
If $S$ is not minimal then $e(S)>0$ and hence any $\sigma\in \Aut_\#(S)$ has fixed points. Using the natural K\"ahler metric with nonpositive sectional curvature on the minimal model $S_{\min}$ of $S$, one sees that $\sigma = \id_S$ as in the proof of \cite[Corollary 0.2]{Liu18}.

Now we may assume that $S$ is minimal. Again, using  the natural K\"ahler metric with nonpositive sectional curvature on $S$, $\sigma$ is smoothly homotopic to $\id_S$ through harmonic maps $\sigma_s\colon S\rightarrow S$ with $\sigma_0=\id_S$ and $\sigma_1=\sigma$ by \cite{Har67}. Moreover, for any $p\in S$, the path $s\mapsto \sigma_s(p)$ is a geodesic segment (parametrized proportionally to arc length) with length independent of $p\in S$. The vector field corresponding to the homotopy $\sigma_s$ lifts to $C\times E$ since $\pi\colon C\times E\rightarrow S$ is \'etale. It follows that $\sigma_s$ lifts to $\tilde\sigma_s \colon C\times E\rightarrow C\times E$ for $0\leq s\leq 1$. For any $p\in S$, since $\sigma$ preserves each fibre of $f\colon S\rightarrow B$,  the geodesic $\sigma_s(p)$ from $p$ to $\sigma(p)$ lies in the fibre of $f$ containing $p$. Now one sees that 
\[
\tilde \sigma_s(\tilde b, x) = (\tilde b, x+t_s(\tilde b))
\]
for some morphism $t_s\colon C\rightarrow E$. Thus $\tilde \sigma_s$ and hence $\sigma_s$ is holomorphic for each $0\leq s\leq 1$. It follows that $\sigma\in \Aut^0(S)$.
\end{proof}

\appendix

\section{}

\subsection{$G$ abelian }
Let us consider first the case where  $G=T\times \mu_r$ is  an abelian group.

We consider the following List I, encompassing 20 cases, corresponding to  all  minimal monodromies 
for $G$ abelian (see Lemma \ref{Mon} and Definition \ref{def:min-mon}). Here, 
$\epsilon$ is a generator of $\mu_r$ and 
the type of the monodromy $(\alpha_1,\beta_1,\ldots, \alpha_h,\beta_h; \gamma_1,\ldots, \gamma_k)$ is defined as the sequence $(h; m_1, \dots, m_k)$, where the $m_j$' s are
the orders of the local monodromies $t_j : = Mon (\ga_j)$.

\begin{enumerate}
    \item $T=(\mathbb Z/3) t$, $r=3$, and the image of the monodromies are:
        \begin{itemize}
            \item[(I-1)] $(\epsilon, 0; t, 2t)$, type: (1;3,3) BIG
            \item[(I-2)] $(\epsilon, t, 0,0)$,  type: (2;-)

            \item[(II-2)] $(\epsilon, 0, t,0)$,  type: (2;-)

            \item[(II-3)] $(\epsilon, t, t,0)$  or $(\epsilon, t, 2t,0)$,    type: (2;-)
        \end{itemize}
        
    \item $T=(\mathbb Z/2) t$, $r=4$, and the image of the monodromies are:
        \begin{itemize}
            \item[(I-1)] $(\epsilon, 0; t,t)$, type: (1;2,2) BIG
            \item[(I-2)] $(\epsilon, t, 0,0)$,  type: (2;-) 
             \item[(II-2)] $(\epsilon, 0, t,0)$,  type: (2;-)

            \item[(II-3)] $(\epsilon, t, t,0)$,   type: (2;-)
        \end{itemize}
        
    \item $T=(\mathbb Z/2) t \oplus (\mathbb Z/2) s$, $r=2$, and the image of the monodromies are:
        \begin{itemize}
            \item[(I-1)] $(\epsilon, 0; t,s,t +s)$,  type: (1;2,2,2) BIG
            \item[(I-3)] $(\epsilon, t; s,s)$, type: (1;2,2) BIG
            \item[(II)] $(\epsilon, t, s,0)$ or $(\epsilon, t, s,t)$, type: (2;-)
            \item[(II-1)] $(\epsilon, 0, t,s)$, type: (2;-)
            \item[(III)] $(\epsilon, t, t,0;s,s)$ or 
            $(\epsilon, 0, t,0;s,s)$, type: (2;2,2)
        \end{itemize}
        
    \item $T=(\mathbb Z/2) t$, $r=2$, and 
    image of the monodromies are:
        \begin{itemize}
            \item[(I-1)] $(\epsilon, 0; t,t)$,  type: (1;2,2)
            \item[(I-2)] $(\epsilon, t,0,0)$, type: (2;-) 
             \item[(II-2)] $(\epsilon, 0, t,0)$,  type: (2;-)
            \item[(II-3)] $(\epsilon, t, t,0)$,    type: (2;-)
        \end{itemize}
\end{enumerate}

\begin{theorem}\label{Magma}

Let $G=T\times \mu_r$ be an abelian group, and assume that the monodromy is one of List I. Then the following holds:

\begin{enumerate}

 \item If  $T=(\mathbb Z/3) t$, $r=3$ and  the monodromy 
is 

 i) $(\epsilon, t, 0,0)$ (Case I-2), or 

ii) $(\epsilon, t, a_2 ,0)$ ($a_2\in T\setminus\{0\}$, Case II-3), 

then $\Aut_{\mathbb Z}(S)=\mathbb Z/3$.

    \item If  $T=(\mathbb Z/2) t$, $r=4$ and  the monodromy 
is 

i) $(\epsilon, 0; t,t)$ (Case I-1, big), or 

ii) $(\epsilon, t, 0,0)$ (Case I-2), or

 iii) $(\epsilon, t, t,0)$ (Case II-3), 
 
 then $\Aut_{\mathbb Z}(S)=\mathbb Z/2$.

 \item If  $T=(\mathbb Z/2) t$, $r=2$ and  the monodromy 
is

 i) $(\epsilon, t, 0,0)$ (Case I-2) or 
 
 ii) $(\epsilon, t, t,0)$ (Case II-3) 
 
 then $\Aut_{\mathbb Z}(S)\subseteq \mathbb Z/2$.

\item In  all other cases $\Aut_{\mathbb Z}(S)$ is trivial.

\end{enumerate}

\end{theorem}

\begin{proof}
The MAGMA \cite{MAGMA} script below shows that there are non-trivial elements in $\Aut_{\QQ}(S)$ acting trivially on $H_1(S,\ZZ)$ (cf. Criterion \ref{tips}) only 
only for the  cases listed in (1), (2), (3) and for $G=\ZZ/2\times \mu_4$, Case (II-2).

In case (1-i) indeed, as explained in Subsection \ref{main-ex}, the cohomology of $S$ is torsion free, hence
$\Aut_{\QQ}(S)= \Aut_{\ZZ}(S)$: then we recall that if $h \geq 2$ (see (4) of Theorem \ref{thm2}) then only the translations are elements of $\Aut_{\QQ}(S)$.

The same occurs also for  (1-ii), (2-ii) and (2-iii).

Instead, for  (3-i) and (3-ii), every element of $G$ induces an automorphism acting trivially on $H_1(S,\ZZ)$ (which has torsion) 
but  the non-translations act non-trivially on 
$H^2(S,\mathbb Z)$, so $\Aut_{\mathbb Z}(S) \subseteq T=\mathbb Z/2$.

 Finally, if $T=(\mathbb Z/2) t$, $r=4$ and  the monodromy 
is (I-1) or (II-2) the scripts show that only the element $\epsilon^2:z\mapsto -z$ induces an automorphism acting trivially on $H_1(S,\mathbb Z)$.
In the latter case, we have a non-translation and $h=2$, so this automorphism acts non-trivially on $H^2(S,\ZZ)$.
In the former case, by
 Lemma \ref{Tcyclic} the automorphism acts trivially on 
$H^2(S,\mathbb Z)$ as well, so $\Aut_{\mathbb Z}(S)$ has order 2.
\end{proof}

\subsection{Sporadic case $G=(\mathbb Z/2)^2\rtimes \mu_4$}
Here we consider the case where $G=(\mathbb Z/2)^2\rtimes \mu_4$, and assume that the  monodromy 
images of the geometric basis elements $(\alpha_1,\beta_1,\ldots, \alpha_h,\beta_h; \gamma_1,\ldots, \gamma_k)$
are as follows, where $\epsilon$ is a generator of $\mu_4$ and $T=(\mathbb Z/2) t\oplus (\mathbb Z/2) s$.

We treat a List II consisting of  11 cases, including all the minimal monodromies for $h=1,2$ (see Lemma \ref{sporadic-mon} and Remark \ref{rem:list_sporadic})
\begin{enumerate}
\item[(IV-2)] $(\epsilon, 0; s,s)$,  type (1;2,2)
 
 \item[(V-1)]$(\epsilon, t; t+s)$,  type (1;2)

\item[(V*-2)] $(\e,0,0,t; s,s)$, type (2;2,2)

\item[(V**-2)]  $(\e,0,0,t; t+s,t+s)$, type (2;2,2) 

\item[(V***-1)]  $(\e,t,0,t; t+s)$, type (2;2)
 
\item[(VI)] $(\epsilon, 0,  t,s)$,  type (2;-)

\item[(VI*)] $(\epsilon, t+s, t,0)$,  type (2;-)

  \item[(VI**)] $(\epsilon, t+s, t,s)$,  type (2;-)

\item[(VIII)] $(\epsilon, 0, t,0)$,  type (2;-)

\item[(VII)] $(\epsilon, 0, t+s,0,s,0)$,  type (3;-)

\item[(VII*)] $(\epsilon, 0, t,0,s,0)$,  type (3;-)

\end{enumerate}

\begin{theorem}\label{Magma_Sporadic_cases}
Let $G=(\mathbb Z/2)^2\rtimes \mu_4$, and assume that the  monodromy is one of List II. Then $\Aut_{\mathbb Z}(S)$ is trivial.
\end{theorem}

\begin{proof}
    The MAGMA script below shows that in all cases  there are no non-trivial elements in $\Aut_{\mathbb Q}(S)$ acting trivially on $H_1(S,\mathbb Z)$, except in Case (VI*) and Case (VI**).
    
    In Case (VI*) and Case (VI**)  the element $z\mapsto -z+t+s$ (respectively $z\mapsto -z$) is the unique non-trivial element in $\Aut_{\mathbb Q}(S)$ acting trivially on $H_1(S,\mathbb Z)$, but  it is  a non-translation and $h=2$, so it acts non-trivially on $H^2(S,\ZZ)$  by (4) of Theorem \ref{thm2}.
\end{proof}

The above examples, although not providing a full classification, suggest that it may be difficult to find a 
case where 
$\Aut_{\mathbb Z}(S) \cong (\ZZ/2)^2.$

\section{MAGMA script}

The script can be run at \url{http://magma.maths.usyd.edu.au/calc/}.\\

\begin{code_magma}
/* Input: i) the group: G;  ii) the monodromy images of pi_orb(B)-> G: mon;  
iii) the genus of the quotient curve B: h;

Output: the number of elements of Z(G)=Aut_Q(S), acting trivially on H_1(S,Z)*/

// First of all, given the monodromy images, 
// we construct the group pi^orb and the monodromy pi^orb-->>G

Orbi:=function(gr,mon, h)
F:=FreeGroup(#mon);  Rel:={}; G:=Id(F);
for i in {1..h} do  G:=G*(F.(2*i-1)^-1,F.(2*i)^-1); end for;
for i in {2*h+1..#mon} do G:=G*F.(i); Include(~Rel,F.(i)^(Order(mon[i])));   end for;
Include(~Rel,G); P:=quo<F|Rel>;
return P, hom<P->gr|mon>;
end function;

// MapProd computes given two maps f,g:A->B the map product
// induced by the product on B

MapProd:=function(map1,map2)
seq:=[];
A:=Domain(map1); 
B:=Codomain(map1);
if Category(A) eq GrpPC then 
	n:=NPCgens(A);
	else n:=NumberOfGenerators(A); 
end if;
for i in [1..n] do Append(~seq, map1(A.i)*map2(A.i)); end for;
return hom<A->B|seq>;
end function;

TrivialActionH1:=function(G,mon, h)
// First of all we construct the group pi^orb and the monodromy pi^orb-->>G

PiOrb,f:=Orbi(G,mon,h);

// We compute a set of generators U for ker(f)=pi_1(C),
 //using the Reidemeister-Schreier method

R:=[PiOrb!1]; 
for g in {g: g in G| g ne G!1} do Append(~R,g@@f); end for;  
// R is a set of representative of the  cosets of ker(f)=pi_1(C)<PiOrb
U:={};
for r in R do
for t in Generators(PiOrb) do 
	h:=r*t;
	if exists(k){s:s in R | f(s) eq f(h)} then
		Include(~U,h*k^-1);
end if; end for; end for;
	
Pi1:=sub<PiOrb| U>; // Pi_1 of C
H1,t:=AbelianQuotient(Pi1); //H_1(C,Z)

/* The group of coinvariants is H_1(C,Z) modulo the subgroup generated by 
{g u g^-1u^-1},
where u is a generator of H_1(C,Z) and g are lifts of generators of G.
We do it in 2 steps, so the presentation given by MAGMA is simpler */

U_new:={}; //the relations for the coinvariants
for u in U  do 
for g in Generators(G) do
	Include(~U_new,t((g@@f)*u*(g@@f)^-1*u^-1));
end for; end for;

H1G, s:=quo<H1|U_new>;//H_1(C,Z)_G 
triv:={G!1};
H:={g: g in Center(G)| g ne G!1}; 

for g in H do 
	l1:=g@@f;  act:={};
	for gen in G do
		l2:=gen@@f; Include(~act,s(t((l1,l2))));
	end for; 
	// check if g in Z(G)=Aut_Q(S), acts trivially on  H_1(S,Z)
	if #act eq 1 then Include(~triv,g); end if;
end for; 

return #sub<G|triv>;
end function;

/* Input: i) the group: G;  ii) the monodromy images of pi_orb(B)-> G: mon;  
iii) the genus of the quotient curve B: h;
iv) the monodromy images of pi_orb(E)-> G: monE. 
Output: the torsion subgroup of H_1(S,Z) */

Tors_H1S:=function(G,mon1, h1, mon2, h2)
T1,f1:=Orbi(G,mon1, h1);
T2,f2:=Orbi(G,mon2, h2);
T1xT2,inT,proT:=DirectProduct([T1,T2]);
GxG,inG:=DirectProduct(G,G);
Diag:=MapProd(inG[1],inG[2])(G);
f:=MapProd(proT[1]*f1*inG[1],proT[2]*f2*inG[2]);
HH:=Rewrite(T1xT2,Diag@@f); // This is the fundamental group of S=(CxE)/G
return TorsionSubgroup(AbelianQuotient(HH)); //Tors(H_1(S,Z))
end function;
\end{code_magma}

We use the previous script, to check, for which monodromies of List I and List II, which elements act non-trivially on $H_1(S,\ZZ)$. If there are non-trivial elements acting non-trivially on $H_1(S,\ZZ)$, we compute also the torsion subgroup of $H_1(S,\ZZ)$.

\begin{itemize}
    \item  $G=\mathbb Z/3 \times \mu_3$

\begin{code_magma}
G:=SmallGroup(9,2); 
// Z/3xZ/3 t=G.1=(1,0), e=G.2=(0,1) 
monE:=[G.2,G.1*G.2,G.1^2*G.2]; //[e, (t,e), (2t,e)]

mon:=[G.2,G!1,G.1,G.1^2];  h:=1; // I-1, BIG
TrivialActionH1(G,mon,h);
> 1

mon:=[G.2,G.1,G!1,G!1]; h:=2; // I-2
TrivialActionH1(G,mon,h);
> 9 // whole group acts trivially on H_1(S,Z)
Tors_H1S(G,mon,h,  monE,0);
> Abelian Group of order 1

mon:=[G.2,G!1,G.1,G!1]; h:=2; // II-2
TrivialActionH1(G,mon,h);
> 1

mon:=[G.2,G.1,G.1,G!1]; h:=2; // II-3 v1
TrivialActionH1(G,mon,h);
> 9 // whole group acts trivially on H_1(S,Z)
Tors_H1S(G,mon,h,  monE,0);
> Abelian Group of order 1

mon:=[G.2,G.1,G.1^2,G!1]; h:=2; // II-3 v2
TrivialActionH1(G,mon,h);
> 9 // whole group acts trivially on H_1(S,Z)
Tors_H1S(G,mon,h,  monE,0);
> Abelian Group of order 1
\end{code_magma}

\item $G=\mathbb Z/2 \times \mu_4$
\begin{code_magma}
G:=SmallGroup(8,2);
 // Z/2xZ/4  t=G.2=(1,0), e=G.1=(0,1)
 monE:=[G.1,G.1*G.2, G.1^2*G.2]; //[e, (t,e), (t,e^2)] 
 
mon:=[G.1,G!1,G.2,G.2]; h:=1; // I-1, BIG
TrivialActionH1(G,mon,h);
> 2 // G.3=G.1^2= e^2
Tors_H1S(G,mon,h,  monE,0);
>  Z/2 + Z/2

mon:=[G.1,G.2,G!1,G!1]; h:=2;  // I-2
TrivialActionH1(G,mon,h);
> 8 // whole group acts trivially on H_1(S,Z)
Tors_H1S(G,mon,h,  monE,0);
> Abelian Group of order 1

mon:=[G.1,G!1,G.2,G!1]; h:=2; // II-2
TrivialActionH1(G,mon,h);
> 2 // G.3=G.1^2= e^2
Tors_H1S(G,mon,h,  monE,0);
>  Z/2 

mon:=[G.1,G.2,G.2,G!1]; h:=2; // II-3
TrivialActionH1(G,mon,h);
> 8 // whole group acts trivially on H_1(S,Z)
Tors_H1S(G,mon,h,  monE,0);
> Abelian Group of order 1
\end{code_magma}

\item $G=(\mathbb Z/2)^2 \times \mu_2$
\begin{code_magma}
G:=SmallGroup(8,5); 
// (Z/2xZ/2)xZ/2  t=G.1=(1,0,0), s=G.2=(0,1,0), e=G.3=(0,0,1) 
monE:=[G.3,G.1*G.3,G.2*G.3,G.1*G.2*G.3];//[e, (t,e), (s,e), (t+s,e)]

mon:=[G.3,G!1,G.1,G.2,G.1*G.2]; h:=1; // I-1 BIG
TrivialActionH1(G,mon,h);
> 1

mon:=[G.3,G.1,G.2,G.2]; h:=1; // I-3 BIG
TrivialActionH1(G,mon,h);
> 1

mon:=[G.3,G.1,G.2,G!1]; h:=2; // II v1
TrivialActionH1(G,mon,h);
> 1

mon:=[G.3,G.1,G.2,G.1]; h:=2; // II v2
TrivialActionH1(G,mon,h);
> 1

mon:=[G.3,G!1,G.1,G.2]; h:=2; // II-1
TrivialActionH1(G,mon,h);
> 1

mon:=[G.3,G.1,G.1,G!1,G.2,G.2]; h:=2; // III v1
TrivialActionH1(G,mon,h);
> 1

mon:=[G.3,G!1,G.1,G!1,G.2,G.2]; h:=2;// III v2
TrivialActionH1(G,mon,h);
> 1
\end{code_magma}

\item $G=\mathbb Z/2 \times \mu_2$
\begin{code_magma}
G:=SmallGroup(4,2);
// Z/2xZ/2 t=G.1=(1,0), e=G.2=(0,1)
monE:=[G.2,G.2,G.1*G.2,G.1*G.2]; // [e,e,(t,e),(t,e)]

mon:=[G.2,G!1,G.1,G.1]; h:=1; // I-1
TrivialActionH1(G,mon,h);
> 1

mon:=[G.2,G.1,G!1,G!1]; h:=2; // I-2
TrivialActionH1(G,mon,h);
> 4 // whole group acts trivially on H_1(S,Z)
Tors_H1S(G,mon,h,  monE,0);
>  Z/2

mon:=[G.2,G!1,G.1,G!1]; h:=2; // II-2
TrivialActionH1(G,mon,h);
> 1

mon:=[G.2,G.1,G.1,G!1]; h:=2; // II-3
TrivialActionH1(G,mon,h);
> 4 // whole group acts trivially on H_1(S,Z)
Tors_H1S(G,mon,h,  monE,0);
>  Z/2
\end{code_magma}

\item $G=(\mathbb Z/2\mathbb Z)^2\rtimes \mu_4$

\begin{code_magma}
G:=SmallGroup(16,3); 
// epsilon= G.1, t=G.2, s=G.2*G.3, t+s=G.3, epsilon^2=G.4
monE:=[G.2*G.4,G.2*G.1,G.1]; //[(t,e^2),(t,e),e]

mon:=[G.1,G!1,G.2*G.3,G.2*G.3]; h:=1; // IV-2
TrivialActionH1(G,mon,h);
> 1

mon:=[G.1,G.2,G.3]; h:=1; // V-1
TrivialActionH1(G,mon,h);
> 1

mon:=[G.1,G!1,G!1,G.2,G.2*G.3,G.2*G.3]; h:=2; // V*-2
TrivialActionH1(G,mon,h);
> 1

mon:=[G.1,G!1,G!1,G.2,G.3,G.3]; h:=2; // V**-2
TrivialActionH1(G,mon,h);
> 1

mon:=[G.1,G.2,G!1,G.2,G.3]; h:=2; // V***-1
TrivialActionH1(G,mon,h);
> 1

mon:=[G.1,G!1,G.2,G.2*G.3]; h:=2; // VI
TrivialActionH1(G,mon,h);
> 1

mon:=[G.1,G.3,G.2,G!1]; h:=2; // VI*
TrivialActionH1(G,mon,h);
> 2 // G.3*G.4 = z-> -z+t+s acts trivially on H1(S,Z)
Tors_H1S(G,mon,h,  monE,0);
> Abelian Group of order 1

mon:=[G.1,G.3,G.2,G.2*G.3]; h:=2; // VI**
TrivialActionH1(G,mon,h);
>2 // G.4 = z-> -z acts trivially on H1(S,Z)
Tors_H1S(G,mon,h,  monE,0);
> Abelian Group of order 1

mon:=[G.1,G!1,G.2,G!1]; h:=2; // VIII
TrivialActionH1(G,mon,h);
> 1

mon:=[G.1,G!1,G.3,G!1, G.2*G.3,G!1]; h:=3; // VII
TrivialActionH1(G,mon,h);
> 1

mon:=[G.1,G!1,G.2,G!1, G.2*G.3,G!1]; h:=3; // VII*
TrivialActionH1(G,mon,h);
> 1 
\end{code_magma}

\end{itemize}

\end{document}